\documentclass[11pt,final]{article}%
\usepackage[textwidth=6.0in,textheight=9.0in,centering]{geometry}
\usepackage{amsxtra,amscd}
\usepackage{amsmath}
\usepackage{amsfonts}
\usepackage{amssymb}
\usepackage{graphicx}
\usepackage[amsmath,hyperref,thmmarks]{ntheorem}
\usepackage{url}
\usepackage{verbatim}
\usepackage{mathptmx}
\usepackage[scaled=0.92]{helvet}
\usepackage[notref,notcite]{showkeys}
\usepackage[all,cmtip]{xy}
\usepackage{diagxy}
\usepackage{natbib}
\usepackage{paralist}
\usepackage{sectsty}%
\theoremnumbering{arabic}
\theoremheaderfont{\scshape}
\RequirePackage{latexsym}
\newtheorem{X}{X}[section]
\theoremseparator{.}

\newtheorem{lrconjecture}[X]{Conjecture LR+}
\newtheorem{spconjecture}[X]{Special-Points Conjecture}
\newtheorem{icconjecture}[X]{Integral Comparison Conjecture}

\newtheorem{lemma}[X]{Lemma}
\newtheorem{proposition}[X]{Proposition}
\newtheorem{theorem}[X]{Theorem}

\theorembodyfont{\upshape}
\newtheorem{definition}[X]{Definition}
\newtheorem{example}[X]{Example}

\newtheorem{question}[X]{Question}
\newtheorem{remark}[X]{Remark}
\newtheorem{plain}[X]{}

\theoremstyle{nonumberplain}
\theoremsymbol{\ensuremath{_\Box}}
\RequirePackage{amssymb}
\newtheorem{proof}{Proof}

\qedsymbol{\ensuremath{_\Box}}
\theorembodyfont{\small}
\newtheorem{aside}[X]{Aside}
\theoremstyle{numberplain}
\theorembodyfont{\footnotesize}

\theoremclass{LaTeX}
\numberwithin{equation}{subsection}
\expandafter\let\csname equation*\expandafter\endcsname\csname gather*\endcsname
\expandafter\let\csname endequation*\expandafter\endcsname\csname endgather*\endcsname

\newcommand{\bquote}{\begin{quote}}
\newcommand{\equote}{\end{quote}}
\newcommand\babstract{\begin{abstract}}
\newcommand\eabstract{\end{abstract}}

\setdefaultitem{$\diamond$\hspace{0.07in}}{}{}{}

\sectionfont{\Large\bfseries\centering}
\subsectionfont{\large\bfseries}
\paragraphfont{\itshape\mdseries}
\let\subsubsection=\paragraph

\let\cite=\citealt
\renewcommand{\theequation}{\arabic{equation}}

\renewcommand{\theenumi}{\alph{enumi}}

\renewcommand{\theenumii}{\roman{enumii}}

\usepackage{paralist}

\newcommand{\ad}{\operatorname{ad}}

\newcommand{\Aut}{\operatorname{Aut}}

\newcommand{\End}{\operatorname{End}}

\newcommand{\Gal}{\operatorname{Gal}}
\newcommand{\GL}{\operatorname{GL}}

\newcommand{\GSp}{\operatorname{GSp}}

\newcommand{\Hom}{\operatorname{Hom}}

\newcommand{\id}{\operatorname{id}}

\newcommand{\im}{\operatorname{Im}}  

\newcommand{\Isom}{\operatorname{Isom}}
\newcommand{\Ker}{\operatorname{Ker}}

\newcommand{\Lie}{\operatorname{Lie}}

\newcommand{\ord}{\operatorname{ord}}

\newcommand{\Res}{\operatorname{Res}}
\newcommand{\Sh}{\operatorname{Sh}}

\newcommand{\SL}{\operatorname{SL}}
\newcommand{\SO}{\operatorname{SO}}
\newcommand{\SP}{\operatorname{Sp}} 

\newcommand{\Spec}{\operatorname{Spec}}

\newcommand{\Sym}{\operatorname{Sym}}

\newcommand{\Tr}{\operatorname{Tr}}

\def\1{{1\mkern-7mu1}}  

\newcommand{\CM}{\mathsf{CM}}
\newcommand{\Hdg}{\mathsf{Hdg}}

\newcommand{\LCM}{\mathsf{LCM}}
\newcommand{\LMot}{\mathsf{LMot}}

\newcommand{\MF}{\mathsf{MF}}
\newcommand{\Mod}{\mathsf{Mod}}
\newcommand{\Mot}{\mathsf{Mot}}

\newcommand{\Rep}{\mathsf{Rep}}

\newcommand{\Vc}{\mathsf{Vec}}

\newcommand\Filt{{\mathrm{Filt}}}

\begin{document}

\title{Points on Shimura varieties over finite fields: the conjecture of Langlands
and Rapoport}
\author{James S. Milne}
\date{November 11, 2009; v4.00}
\maketitle

\begin{abstract}
We state an improved version of the conjecture of Langlands and Rapoport, and
we prove the conjecture for a large class of Shimura varieties. In particular,
we obtain the first proof of the (original) conjecture for Shimura varieties
of PEL-type.

\end{abstract}
\tableofcontents

\clearpage\pagestyle{headings}\thispagestyle{plain}

\section*{Introduction}

\addcontentsline{toc}{section}{Introduction}

\hfill\begin{minipage}{4.5in}
\emph{
Das Problem der Fortsetzbarkeit der Hasse-Weil-Zeta-Funktionen und allgemeiner
der motivischen $L$-Funktionen ist nach wie vor ein zentrales Problem der
Zahlentheorie. Es wird oft in zwei Probleme aufgeteilt \ldots. Es ist erstens zu zeigen, dass jede motivische $L$-Funktion
gleich einer automorphen $L$-Funktion ist, und zweitens, dass jede automorphe
$L$-Funktion fortsetzbar ist. Beide Probleme sind in herzlich wenigen
F\"{a}llen gel\"{o}st und dann nur dank der Bem\"{u}hungen vieler Mathematiker
\"{u}ber lange Zeit. Nach den abelschen Variet\"{a}ten sind in arithmetischer
Hinsicht die Shimuravariet\"{a}ten wohl die zug\"{a}nglichsten, und diese
Arbeit soll ein Beitrag zum ersten Problem f\"{u}r die ihnen zugeordneten
motivschen $L$-Funktionen sein.\footnotemark}\\
\hspace*{2in} Langlands and Rapoport 1987, p113.
\end{minipage}\footnotetext{The problem of analytically continuing Hasse-Weil
zeta functions, and more generally motivic $L$-functions, is one of the
central problems in number theory. It is often divided into two problems\ldots
. The first is to show that every motivic $L$-function is an automorphic
$L$-function, and the second is to show that every automorphic $L$-function
can be analytically continued. Both problems have been solved in mightily few
cases, and then only thanks to the efforts of many mathematicians over a long
period. After abelian varieties, Shimura varieties are, from the arithmetical
point of view, the most approachable, and this work is a contribution to the
first problem for their associated motivic $L$-series.} \bigskip

Shimura varieties arose out of the study of automorphic functions, and are
defined by a reductive group $G$ and additional data $X$. In order to show
that the zeta function of a Shimura variety is an automorphic $L$-function,
one must find a group-theoretic description of the points of the variety with
values in a finite field, and then apply a combinatorial argument involving
the stabilized Arthur-Selberg trace formula and the fundamental lemma.
\citet{langlands1976} gives a general conjectural description of the points.
For the Shimura varieties attached to totally indefinite quaternion algebras I
showed that this conjecture can be proved using only the results of Weil,
Tate, and Honda on abelian varieties over finite fields (\cite{milne1979,
milne1979e}). From this, it follows that the zeta functions of these varieties
are automorphic (see \cite{casselman1979}). Similar results were obtained
using Honda-Tate theory for other quaternionic Shimura varieties by
\citet{reimann1997}, and for simple Shimura varieties of PEL-types A and C by \citet{kottwitz1992}.

Although these results have important consequences, as Langlands himself
pointed out, his original conjecture is inadequate. Typically in the theory of
Shimura varieties, one proves a statement for some small class of Shimura
varieties and extends it to a much larger class through the intermediary of
connected Shimura varieties. Langlands's conjecture is too imprecise for this
approach to work. Moreover, it groups together objects that are only locally
isomorphic whereas one should have a finer statement in which globally
nonisomorphic objects are distinguished. Finally, Langlands states his
conjecture in terms of an embedding of a group into $G(\mathbb{A}_{f}^{p})$,
and this embedding is not sufficiently precisely defined to permit passage to
the combinatorial part of the argument in the general case.

In their fundamental paper, \citet{langlandsR1987} use a \textquotedblleft
Galois-gerbe\textquotedblright, which is conjecturally the groupoid attached
to a fibre functor on the category of motives over a finite field, to give a
precise conjectural description of the points on the reduced variety (ibid.
5.e). This conjecture removes the inadequacies of Langlands's original
conjecture, and is a much deeper statement. In particular, it is not
susceptible to proof, even for Shimura varieties of
PEL-type,\footnote{Although Shimura varieties of PEL-type are very important,
they are very special; see \cite{deligne1971}, p123.} by using only Honda-Tate
theory.\footnote{Except for some very special Shimura varieties --- loosely
speaking, those for which there is no $L$-indistinguishability --- for
example, those defined by quaternion algebras; see \cite{reimann1997},
pp52-59.} Recall that this theory provides a list of the isogeny classes of
abelian varieties over a finite field and determines the isomorphism class of
the endomorphism algebra attached to each class. In Section 6 of their paper,
Langlands and Rapoport proved their conjecture for simple Shimura varieties of
PEL-types A and C assuming

\begin{enumerate}
\item[(c1)] the Hodge conjecture for CM abelian varieties (over $\mathbb{C}$),

\item[(c2)] the Tate conjecture for abelian varieties over a finite field, and

\item[(c3)] Grothendieck's standard conjectures for abelian varieties over
finite fields.
\end{enumerate}

\noindent These statements allowed them to obtain a precise description of the
\emph{category} of abelian motives over a finite field (and therefore of the
category of abelian varieties) as a polarized tannakian category with the
standard fibre functors.

After their paper, two problems remained:

\begin{itemize}
\item remove the three assumptions (c1,c2,c3) from their proof;

\item extend the proof to all Shimura varieties.
\end{itemize}

Concerning the first problem, in \cite{milne1999lm, milne2002}, I proved that
(c1) implies both (c2) and (c3). Thus, the Hodge conjecture for CM abelian
varieties alone suffices for the result of Langlands and Rapoport. More
recently (\cite{milne2000}, \cite{milne2009}) I showed that a much weaker
statement, namely, the rationality conjecture for CM abelian varieties, has
many of the same consequences as the Hodge conjecture for CM abelian
varieties; in particular, it suffices for the above proof of Langlands and Rapoport.

Concerning the second problem, in \cite{milne1994s} I gave a partial heuristic
derivation of the conjecture of Langlands and Rapoport assuming the existence
of a sufficiently good theory of motives in mixed characteristic, and in the
original version of this article (\cite{milne1995}), I examined what was
needed to turn the heuristic argument into a proof. I was led to state two
conjectures, one concerning lifts of special points and one concerning a
comparison of integral cohomologies. \citet{vasiu2003cm} has announced a
results on the first conjecture, and both \citet{vasiu2003}\ and
\citet{kisin2007, kisin2009}\ have announced results on the second conjecture.

This progress has encouraged me to rewrite my 1995 article. Let $\mathbb{F}{}$
be an algebraic closure of the field $\mathbb{F}{}_{p}$ of $p$ elements. After
some preliminaries on tannakian categories in Section 1, I construct in
Section 2 a category $\Mot(\mathbb{F})$ of abelian motives over $\mathbb{F}$
having most of the properties that the category of Grothendieck motives over
$\mathbb{F}{}$ would have if the three conjectures (c1,c2,c3) were known. In
more detail:

\begin{itemize}
\item The category $\Mot(\mathbb{F}{})$ is constructed as a quotient of the
category $\CM(\mathbb{Q}^{\mathrm{al}})$ of CM-motives over $\mathbb{Q}%
{}^{\mathrm{al}}$. According to the theory of quotient tannakian categories in
\cite{milne2007}, to construct such a quotient, we need a fibre functor
$\omega_{0}$ on a certain subcategory of $\CM(\mathbb{Q}^{\mathrm{al}})$. The
proof of the existence of $\omega_{0}$ makes use of, among other things, the
main result of \cite{wintenberger1991} (which gives an explicit description of
the functor sending a CM-motive to its associated filtered Dieudonn\'{e} module).

\item The proof that $\Mot(\mathbb{F}{})$ has the correct fundamental group
uses the main ideas of \cite{milne1999lm} (which proves that (c1) implies (c2)).

\item The proof that the polarization on $\CM(\mathbb{Q}^{\mathrm{al}})$
descends to a polarization on $\Mot(\mathbb{F})$ uses the main ideas of
\cite{milne2002} (which proves that (c1) implies (c3)).
\end{itemize}

In Section 3, I use the category $\Mot(\mathbb{F}{})$ to give what I believe
to be the \textquotedblleft right\textquotedblright\ statement of the
conjecture of Langlands and Rapoport (henceforth called the Conjecture LR+;
see below for a discussion of the various forms of the conjecture). It
attaches to each Shimura $p$-datum $(G,X)$ a set $\mathcal{L}(G,X)$ with
operators, and the conjecture states that there should be a functorial
isomorphism $\mathcal{L}(G,X)\rightarrow\Sh_{p}(G,X)(\mathbb{F})$.

In the fourth section, I explain (following \cite{milne1994s}) how to realize
many Shimura varieties in characteristic zero as moduli varieties of abelian
\emph{motives}. One problem in tackling the Langlands-Rapoport conjecture is
that the level structure at $p$ in characteristic zero is stated in terms of
\'{e}tale cohomology, whereas the level structure at $p$ over the finite field
is stated in terms of the crystalline cohomology. In order to pass from one to
the other, we need the integral comparison conjecture (first stated in
\cite{milne1995}) which says that a Hodge class on an abelian variety with
good reduction is integral for the de Rham cohomology if it is integral for
the $p$-adic \'{e}tale cohomology. Since proofs of enough of this conjecture
for our purposes have been announced by both \citet{vasiu2003} and
\citet{kisin2007, kisin2009}, I shall assume it for the remainder of this introduction.

Another obstacle is that the conjecture of Langlands and Rapoport implicitly
implies that points of the Shimura variety with coordinates in a finite field
lifts to special points in characteristic zero (up to isogeny). I call this
statement the special-points conjecture. For simple Shimura varieties of PEL
type, it was proved by \citet{zink1983}, and more general results been
announced by \citet{vasiu2003cm}.

In the final two sections of the paper, I prove that, for any Shimura
$p$-datum of Hodge type (that is, embeddable in a Siegel $p$-datum), there is
a canonical equivariant map $\mathcal{L}(G,X)\rightarrow\Sh_{p}(\mathbb{F}{}%
)$. For an appropriate integral model the map is injective, and it is
surjective if (and only if) the special-points conjecture is true. In
particular, the Conjecture LR+ (a fortiori, the original conjecture of
Langlands and Rapoport) is proved for Shimura varieties of PEL-type (by Zink's
result). I also discuss how to extend the proofs of the Conjecture LR+ to
other Shimura varieties, including many that are not moduli varieties, not
even conjecturally; cf. Pfau 1993 1996a,b. Moreover, I discuss what is needed
to extend the proof to all Shimura varieties of abelian type, and perhaps to
all Shimura varieties.

\subsection{The various forms of the conjecture of Langlands and Rapoport
(good reduction case)}

In an attempt to reduce the confusion surrounding the statement of the
conjecture, I list its main forms. To avoid overloading the exposition, I
state the conjectures only for Shimura varieties whose weight is defined over
$\mathbb{Q}{}$.

\begin{description}
\item[LRo] The original statement is (5.e), p169, of \cite{langlandsR1987}. In
the first three sections of the paper, the authors define a groupoid (die
pseudomotivische Galoisgruppe), and use it to attach to a Shimura $p$-datum a
set $\mathcal{L}(G,X)$ with a Frobenius operator and an action of
$G(\mathbb{A}{}_{f})$. When $G^{\mathrm{der}}$ is simply connected, they
conjecture that this set with operators is isomorphic to the set
$\Sh_{p}(\mathbb{F}{})$ defined by some integral model of the Shimura variety
$\Sh_{p}(G,X)$.\footnote{Langlands and Rapoport inadvertently omitted the
condition that the map be $G(\mathbb{A}{}_{f})$-equivariant --- Langlands has
assured me that this should be considered part of the conjecture. Langlands
and Rapoport also state a conjecture when the Shimura variety has mild bad
reduction, but we are not concerned with that.}

\item[LRm] In \cite{milne1992} I made some improvements to the original
conjecture (ibid. 4.8).

\begin{itemize}
\item Langlands and Rapoport (1987, \S 7) show that their statement of the
conjecture can not be true when $G^{\mathrm{der}}$ is not simply connected. I
modified the statement of the conjecture so that it applies to all Shimura
varieties (when the derived group is simply connected, the modified statement
becomes the original statement because of \cite{langlandsR1987}, Satz 5.3, p173).

\item I defined the notion of a canonical integral model for the Shimura
variety, which is uniquely characterized by having a certain extension
property, and added the requirement that the conjecture hold for that
particular integral model.

\item I added the condition that the isomorphism $\mathcal{L}(G,X)\rightarrow
\Sh_{p}(G,X)(\mathbb{F}{})$ commutes with the actions of $Z(\mathbb{Q}{}_{p})$
where $Z$ is the centre of $G$. With the addition of this condition, I showed
that the conjecture for Shimura varieties with simply connected derived group
implies the conjecture for all Shimura varieties.
\end{itemize}

\item[LRp] \citet{pfau1993,pfau1996a,pfau1996b} pointed out that neither LRo
nor LRm is sufficiently strong to pass from Shimura varieties of Hodge type to
Shimura varieties of abelian type. Specifically, if one assumes that
Conjecture LRo (or LRm) holds for all Shimura varieties of Hodge type, then it
is not possible to deduce that it holds for all Shimura varieties of abelian
type. For that, one needs a \textquotedblleft refined\textquotedblright%
\ conjecture in which the isomorphism $\mathcal{L}(G,X)\rightarrow
\Sh_{p}(G,X)(\mathbb{F}{})$ is required to respect the maps to the sets of
connected components.

\item[LR+] As Deligne pointed out to me, Langlands and Rapoport define only
the isomorphism class of their groupoid. In fact, the groupoid is not
well-defined, even conjecturally (at a minimum it requires the choice of a
fibre functor). In \S 3, I restate the conjecture in terms of the category
$\Mot(\mathbb{F}{})$ defined in \S 2. At present, the construction of this
category also requires a choice, but the possibly-provable rationality
conjecture for CM-abelian varieties (weaker than the Hodge conjecture for CM
abelian varieties) implies that there is a unique preferred choice; moreover,
the choice doesn't affect the construction of $\mathcal{L}(G,X)$. Now that
both objects are well defined, it is possible to require that the isomorphism
$\mathcal{L}(G,X)\rightarrow\Sh_{p}(G,X)(\mathbb{F}{})$ be functorial in
$(G,X)$ --- this is the Conjecture LR+. With the choice of a fibre functor for
$\Mot(\mathbb{F}{})$, Conjecture LR+ implies Conjecture LRp, and so it is
strictly stronger than both LRo and LRm. When one assumes LR+ for all Shimura
varieties of Hodge type, then it is possible to deduce it for all Shimura
varieties of abelian type.
\end{description}

\noindent The conjecture of Langlands and Rapoport is of interest to everyone
working on Shimura varieties. For those interested only in the zeta functions
of Shimura varieties, all that is needed is the integral formula,
\begin{equation}
T(j,f)=\left\vert \Ker^{1}(\mathbb{Q}{},G)\right\vert \sum_{(\gamma_{0}%
;\gamma,\delta)}c(\gamma_{0};\gamma,\delta)\cdot O_{\gamma}(f^{p})\cdot
TO_{\delta}(\phi_{r})\cdot\Tr\xi(\gamma_{0}) \label{eq5}%
\end{equation}
first conjectured in a preliminary form by Langlands, and then by
\nocite{kottwitz1990}Kottwitz (1990, 3.1)\footnote{But only when
$G^{\mathrm{der}}$ is simply connected. There may be some interest in writing
the formula also in the case that the derived group is not simply connected;
that is, in extending \S 5--\S 7 of \cite{milne1992} to that case.}. It is
proved in \cite{milne1992} that Conjecture LRm implies this formula (the
converse, of course, is false).

\subsection{Some history}

The original conjecture of Langlands and Rapoport (LRo) predicted that two
objects, not well-defined, are isomorphic. It is not possible to prove such a
statement without first defining the objects. This led me (in \cite{milne1992}%
) to introduce the notion of a canonical integral model, which ensured that
the set-with-operators $\Sh_{p}(G,X)(\mathbb{F}{})$ is well-defined.

Although Honda-Tate theory suffices to prove Langlands's original conjecture
for some PEL Shimura varieties, it soon became clear (to me and others) that
it was insufficient to prove the Conjecture LRo. Hence there was a need to
obtain a description of the category of abelian varieties over $\mathbb{F}{}$,
or, more generally, of abelian motives, and not just the set of its
isomorphism classes. In 1995, at the time I proved the results in
\cite{milne1999lm}, I thought these results could be used

\begin{quote}
to construct a canonical category of \textquotedblleft
motives\textquotedblright\ over $\mathbb{F}{}$ that has the \textquotedblleft
correct\textquotedblright\ fundamental group, equals the true category of
motives if the Tate conjecture holds for abelian varieties over $\mathbb{F}{}%
$, and canonically contains the category of abelian varieties up to isogeny as
a polarized subcategory (ibid, p47).
\end{quote}

\noindent I applied this statement to investigate the conjecture of Langlands
and Rapoport with the goal of determining what more was needed to prove the
conjecture. I determined (I believe correctly) that the two key technical
results needed are the special points conjecture and integral comparison
conjecture (see below). I wrote this work up as \cite{milne1995} for my own
personal use, but I gave the manuscript to a few people because I wanted to
encourage the experts to work on the two conjectures.

Alas, when I tried to write out the proof of the quoted statement, I found a
gap in my argument (the rationality conjecture!). However my later work has
enabled me to construct a category of motives with the required properties
(but not to prove that it is canonical or that it contains the category of
abelian varieties). This, together with the work of Kisin and Vasiu on the two
conjectures, has encouraged me to return to the topic.

\begin{aside}
A problem one has in working on the conjecture of Langlands and Rapoport is
the misperceptions that exist about the conjecture in the mathematical
community. One factor contributing to this is that the original paper is in
German, and hence inaccessible to most mathematicians, who may also be
deterred by its length (108 pages). Another has been the misstatements in the
literature, most egregiously in Clozel's Bourbaki talk (\cite{clozel1993}),
where he writes (Introduction):\bquote\ldots En collaboration avec Rapoport,
[Langlands] formula ensuite une conjecture pr\'{e}cise (LR1987), qui \'{e}tait
d\'{e}montr\'{e}e modulo une partie des \textquotedblleft conjectures
standard\textquotedblright\ de g\'{e}ometrie alg\'{e}brique.\newline L'objet
de cet expos\'{e} est le travail de Kottwitz (K1992) qui d\'{e}montre
inconditionnellement, pour les vari\'{e}t\'{e}s de Shimura qui d\'{e}crivent
des probl\`{e}mes de modules de vari\'{e}t\'{e}s ab\'{e}liennes munies de
quelque structures \ldots, une reformulation de la conjecture de
Langlands-Rapoport.\smallskip\newline(\ldots In collaboration with Rapoport,
[Langlands] next formulated a precise conjecture (LR1987), which was proved
modulo part of the \textquotedblleft standard conjectures\textquotedblright%
\ in algebraic geometry.\newline The object of this exposition is the work of
Kottwitz (K1992) which proves unconditionally, for those Shimura varieties
describing moduli problems for abelian varieties endowed with some
structures\ldots, a reformulation of the conjecture of
Langlands-Rapoport.)\equote If you believe this, as many mathematicians seem
to judging by their writings, then you will think that the conjecture of
Langlands and Rapoport was proved for (at least) all Shimura varieties of PEL
type by Kottwitz in 1992, and that the general case was proved by Langlands
and Rapoport in 1987 assuming only a part of the standard conjectures. In
fact, \citet{kottwitz1992} proves only the integral formula (\ref{eq5}) (see
p442 of his paper) and only for simple Shimura varieties of PEL types A and C,
while Langlands and Rapoport (1987) prove their conjecture only for simple
Shimura varieties of PEL types A and C, and only assuming the Hodge
conjecture, the Tate conjecture, and the standard conjectures.\footnote{In
diesem Abschnitt wollen wir zeigen, dass f\"{u}r gewisse Shimuravariet\"{a}ten
die Vermutung im \S 5 eine Folge der Identifizierung der pseudomotivischen
Galoisgruppe und der motivischen Galoisgruppe ist, die im \S 4 unter Annahme
der Standardvermutungen und der Tate-Vermutung sowie der Hodge-Vermutung
vorgenommen wurde (ibid. p198).
\par
(In this section we shall show that, for certain Shimura varieties, the
conjecture in \S 5 is a consequence of the identification of the pseudomotivic
Galois group with the motivic Galois group, which was proved in \S 4 under the
assumption of the standard conjectures and the Tate conjecture as well as the
Hodge conjecture.)}
\end{aside}

\subsection*{Notations and conventions}

Throughout $\mathbb{Q}^{\mathrm{al}}$ is the algebraic closure of $\mathbb{Q}$
in $\mathbb{C}$, and $\mathbb{Q}^{\mathrm{cm}}$ is the union of the
CM-subfields of $\mathbb{Q}^{\mathrm{al}}$. For a Galois extension $K/k$, we
let $\Gamma_{K/k}=\Gal(K/k)$. When $K$ is an algebraic closure of $k$, we omit
it from the notation. Complex conjugation on $\mathbb{C}$ and its subfields is
denoted by $\iota$ or $z\mapsto\bar{z}$.

\noindent\begin{minipage}{2.7in}
\hspace{0.2in}We fix a prime $v$ of $\mathbb{Q}^{\mathrm{al}}$ dividing $p$, and we write
$v_{K}$ (or just $v$) for the prime it induces on a subfield $K$ of
$\mathbb{Q}^{\mathrm{al}}$. The completion $(\mathbb{Q}^{\mathrm{al}}%
)_{v}$ of $\mathbb{Q}^{\mathrm{al}}$ at $v$ is algebraically closed, and we
let $\mathbb{Q}_{p}^{\mathrm{al}}$ denote the algebraic closure of
$\mathbb{Q}_{p}$ in $(\mathbb{Q}^{\mathrm{al}})_{v}$. We let $\mathbb{Q}%
_{p}^{\mathrm{un}}$ denote the largest unramified extension of $\mathbb{Q}%
_{p}$ in $\mathbb{Q}_{p}^{\mathrm{al}}$. Its residue field, which we
denote $\mathbb{F}$, is an algebraic closure of $\mathbb{F}_{p}$. We let
$B(\mathbb{F})$ denote the closure of $\mathbb{Q}_{p}^{\mathrm{un}}$ in
($\mathbb{Q}^{\mathrm{al}})_{v}$, and we let $W(\mathbb{F})$ denote its
ring of integers (equal to the ring of Witt vectors with coefficients in
$\mathbb{F}$). The Frobenius map $x\mapsto x^p$ on $\mathbb{F}$
and its lift to $W(\mathbb{F})$ are both denoted by $\sigma$.\end{minipage}\hfill
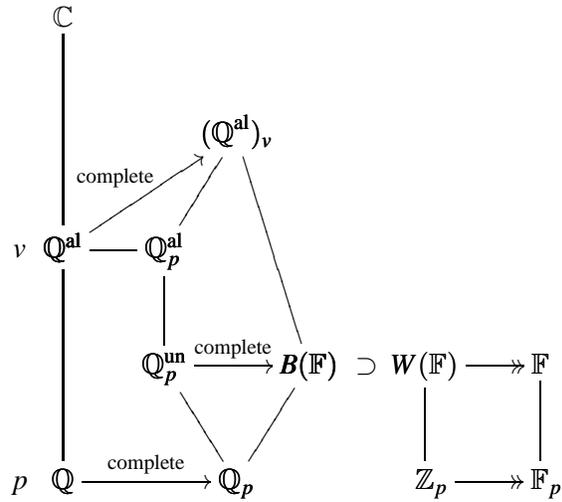
\begin{minipage}{3in}\[\bfig 
\node v(-450,800)[v]
\node p(-450,0)[p]
\node 0(-300,1600)[\mathbb{C}] 
\node a(-300,0)[\mathbb{Q}]
\node b(-300,800)[\mathbb{Q}^{\mathrm{al}}] 
\node c(300,0)[\mathbb{Q}_{p}]
\node d(50,400)[\mathbb{Q}_{p}^{\mathrm{un}}] 
\node e(50,800)[\mathbb{Q}_{p}^{\mathrm{al}}] 
\node f(300,1200)[\mathbb{(Q}^{\mathrm{al}})_{v}]
\node g(550,400)[B(\mathbb{F})] 
\node h(950,0)[\phantom{_p}\mathbb{Z}_{p}]
\node i(950,400)[W(\mathbb{F})] 
\node j(1350,0)[{\mathbb{F}_{\rlap{$\scriptstyle{p}$}}}]
\node k(1350,400)[\mathbb{F}] 
\node m(750,400)[\supset]
\arrow/{-}/[b`e;{}] 
\arrow/{-}/[b`0;{}]
\arrow/{-}/[a`b;{}] 
\arrow[a`c;{\text{complete}}] 
\arrow[b`f;{\text{complete}}] 
\arrow/{-}/[c`d;{}]
\arrow/{-}/[d`e;{}] 
\arrow/{-}/[e`f;{}] 
\arrow/{-}/[c`g;{}] 
\arrow/{-}/[f`g;{}] 
\arrow[d`g;{\text{complete}}]
\arrow/{-}/[h`i;{}] 
\arrow/{->>}/[h`j;{}] 
\arrow/{->>}/[i`k;{}] 
\arrow/{-}/[j`k;{}] 
\arrow/{}/[g`m;{}]
\arrow/{}/[v`b;{}]
\arrow/{}/[p`a;{}] 
\efig\]
\end{minipage}

A reductive group is a smooth affine group scheme whose geometric fibres are
connected reductive algebraic groups. For such a group $G$ over a field,
$G^{\mathrm{der}}$ denotes the derived group of $G$, $Z(G)$ the centre of $G$,
$G^{\mathrm{ad}}\overset{\text{{\tiny def}}}{=}G/Z(G)$ the adjoint group of
$G$, and $G^{\mathrm{ab}}\overset{\text{{\tiny def}}}{=}G/G^{\mathrm{der}}$
the largest commutative quotient of $G$.

For a (pro)torus $T$ over a field $k$, $X^{\ast}(T)$ and $X_{\ast}(T)$ denote
the character group of $T$ and its cocharacter group (characters and
cocharacters defined over some algebraic closure of $k$). The pairing%
\[
\langle\,,\,\rangle\colon X^{\ast}(T)\times X_{\ast}(T)\rightarrow\mathbb{Z}%
\]
is defined by the formula%
\[
(\chi\circ\mu)(t)=t^{\langle\chi,\mu\rangle}\text{, }t\in T(k^{\mathrm{al}}).
\]

For an affine group scheme $G$ over a ring $R$ and an $R$-algebra $S$,
$\Rep_{S}(G)$ denotes the category of representations of $G$ on \emph{finitely
generated projective} $S$-modules (equivalently, flat $S$-modules of finite
presentation). A representation will be denoted $\xi\colon G\rightarrow
\GL(V(\xi))$ or $\xi\colon G\rightarrow\GL(\Lambda(\xi))$ depending on whether
$R$ is a field or not. Thus, $V\colon\xi\rightsquigarrow V(\xi)$ and
$\Lambda\colon\xi\rightsquigarrow\Lambda(\xi)$ are the forgetful fibre
functors $\Rep_{S}(G)\rightarrow\Mod_{S}$. A tensor functor $\Rep_{R}%
(G)\rightarrow\Mod_{S}$ is said to be \emph{exact} if (a) it maps sequences in
$\Rep_{R}(G)$ that are exact as sequences of $R$-modules to exact sequences in
$\Mod_{S}$ and (b) every homomorphism in $\Rep_{R}(G)$ whose image in
$\Mod_{S}$ is an isomorphism is itself an isomorphism.

For a finite extension of fields $k\supset k_{0}$ and an algebraic group $G$
over $k$, $\Res_{k/k_{0}}G$ and $(G)_{k/k_{0}}$ both denote the algebraic
group over $k_{0}$ obtained from $G$ by restriction of scalars. For an
infinite extension $k/k_{0}$, we let $(\mathbb{G}_{m})_{k/k_{0}}$ denote the
protorus $\varprojlim(\mathbb{G}_{m})_{K/k_{0}}$ where $K$ runs over the
finite extensions of $k_{0}$ contained in $k$.

\label{pdatum}By a \emph{Shimura }$p$\emph{-datum}, we mean a reductive
group\footnote{To give a reductive group over $\mathbb{Z}_{(p)}\overset
{\textup{{\tiny def}}}{=}\mathbb{Q}\cap\mathbb{Z}_{p}$ amounts to giving a
reductive group $G_{0}$ over $\mathbb{Q}$, a reductive group $G_{p}$ over
$\mathbb{Z}_{p}$, and an isomorphism $(G_{0})_{\mathbb{Q}_{p}}\rightarrow
(G_{p})_{\mathbb{Q}_{p}}$ (apply \cite{boschLR1990}, 6.2, Proposition D.4,
p147).} $G$ over $\mathbb{Z}_{(p)}$ together with a $G(\mathbb{R})$-conjugacy
class $X$ of homomorphisms $\mathbb{S}\rightarrow G_{\mathbb{R}}$ such that
$(G_{\mathbb{Q}},X)$ satisfies the conditions (SV1), (SV2), and (SV3) of
\cite{milne2005} (equal to the conditions (2.1.1.1), (2.1.1.2), and (2.1.1.3)
of \cite{deligne1979}). Except for the last two subsections of \S 6, we shall
assume that $(G_{\mathbb{Q}},X)$ satisfies (SV4) (the weight is defined over
$\mathbb{Q}{}$) and (SV6) (the connected centre splits over a CM-field).

In general, an object defined over $\mathbb{Q}_{l}$ is denoted by $?_{l}$,
whereas an object over $\mathbb{Q}_{l}$ that comes from an object $?$ over
$\mathbb{Q}$ by extension of scalars is denoted by $?(l)$. We sometimes
abbreviate $S\otimes_{R}?$ to $?_{S}$.

We sometimes use $[x]$ to denote the equivalence class of an element $x$.

We use $\approx$ to denote an isomorphism, and $\simeq$ to denote a canonical
(or given) isomorphism.

We use $l$ to denote a prime of $\mathbb{Q}$, i.e., $l\in\{2,3,\ldots
,p,\ldots\infty\}$, and $\ell$ to denote a prime $\neq p,\infty$.

We let $\mathbb{A}_{f}$ denote the ring of finite ad\`{e}les $(\varprojlim
_{m}\mathbb{Z}/m\mathbb{Z})\otimes_{\mathbb{Z}}\mathbb{Q}$ and $\mathbb{A}%
_{f}^{p}$ the ring of finite ad\`{e}les with the $p$-component omitted.

A diagram of functors and categories is said to commute if it commutes up to a
canonical natural isomorphism.

\section{Tannakian preliminaries}

By a tensor category over a ring $R$ we mean an additive symmetric monoidal
category such that $R=\End(\1)$ for any identity object $\1$
(cf.\ \cite{deligneM1982}). Tensor functors are required to be linear for the
relevant rings. A tensor category over a field is tannakian if it is abelian,
rigid, and admits an $R$-valued fibre functor for some nonzero $k$-algebra
$R$. When the fundamental group of a tannakian category is commutative, we
identify it with an affine group scheme in the usual sense
(cf.\ \cite{deligne1989}, \S 6). For a subgroup $H$ of the fundamental group
of a tannakian category $\mathsf{C}$, we let $\mathsf{C}^{H}$ denote the full
subcategory of objects fixed by $H$ (that is, on which the action of $H$ is trivial).

\subsection{Fibre functors}

Recall that a field $k$ is said to have dimension $\leq1$ if the Brauer group
of every field algebraic over$\ $it is zero (\cite{serre1964}, II \S 3). For
example, a finite field has dimension $\leq1$ and the field $B(\mathbb{F})$
has dimension $\leq1$ (ibid.). For a connected algebraic group $G$ over a
perfect field of dimension $\leq1$, $H^{1}(k,G)=0$ (\cite{steinberg1965}, 1.9).

\begin{proposition}
\label{a1}Let $G$ be a reductive group over a henselian discrete valuation
ring $R$ whose residue field has dimension $\leq1$. Every exact tensor functor
$\omega\colon\Rep_{R}(G)\rightarrow\Mod_{R}$ is isomorphic to the forgetful
functor $\omega^{G}$; moreover, $\Hom^{\otimes}(\omega^{G},\omega)$ is a
principal homogenous space for $G(R)$.
\end{proposition}

\begin{proof}
The tensor isomorphisms from $\omega$ to the forgetful functor form a
$G$-torsor (e.g., \cite{milne1995}, 1.1), which determines an element of
$H^{1}(R,G)$. Because $G$ is of finite type, this fpqc cohomology group can be
interpreted as an fppf group (\cite{saavedra1972}, III 3.1.1.1), and because
$G$ is smooth
\[
H^{1}(R,G)=H^{1}(k,G_{k})
\]
where $k$ is the residue field of $R$ (e.g., \cite{milne1980}, III 3.1). But
$G_{k}$ is connected, and so $H^{1}(k,G_{k})=0$ by Steinberg's theorem, from
which the statement follows.
\end{proof}

Let $G$ be an affine group scheme, flat and of finite type over a ring $R$,
and let $\xi$ be a representation of $G$ on a finitely generated projective
$R$-module $\Lambda(\xi)$. By a tensor on $\xi$ we mean an element of
$\Lambda(\xi)^{\otimes r}\otimes\Lambda(\xi)^{\vee}{}^{\otimes s}$ for some
$r,s$ that is fixed under the action of $G$. Note that a tensor $t$ can
regarded as a homomorphism $R\rightarrow\Lambda(\xi)^{\otimes r}\otimes
\Lambda(\xi)^{\vee}{}^{\otimes s}$ of $G$-modules, and so it defines a tensor
$\omega(t)$ on $\omega(\Lambda(\xi))$ for any tensor functor $\omega$ on
$\Rep_{R}(G)$. A representation $\xi_{0}$ of $G$ together with a family
$(t_{i})_{i\in I}$ of tensors is said to be \emph{defining\/} if, for all flat
$R$-algebras $S$,
\[
G(S)=\{g\in\Aut(S\otimes_{R}\Lambda(\xi_{0}))\mid gt_{i}=t_{i}%
,\text{\textrm{\ for all }}i\in I\}.
\]
In particular, this implies that $\xi_{0}$ is faithful. For conditions under
which a defining representation and tensors exist, see \cite{saavedra1972},
p151. For example, defining tensors exist if $R$ is a discrete valuation ring
and $G$ is a closed flat subgroup of $\GL(\Lambda)$ whose generic fibre is
reductive (by a standard argument, cf. \cite{deligne1982}, 3.1).

\begin{proposition}
\label{a2} Let $G$ be an affine group scheme, flat over a henselian discrete
valuation ring $R$ whose residue field has dimension $\leq1$, and assume that
$(\xi_{0},(t_{i})_{i\in I})$ is defining for $G$.

\begin{enumerate}
\item Consider a finitely generated projective $R$-module $\Lambda$ and a
family $(s_{i})_{i\in I}$ of tensors for $\Lambda$. There exists an exact
tensor functor $\omega\colon\Rep_{R}(G)\rightarrow\Mod_{R}$ such that
\begin{equation}
(\omega(\xi_{0}),(\omega(t_{i}))_{i\in I})=(\Lambda,(s_{i})_{i\in I})
\label{eq1}%
\end{equation}
if and only if there exists an isomorphism $\Lambda(\xi_{0})\rightarrow
\Lambda$ of $R$-modules mapping each $t_{i}$ to $s_{i}$.

\item For any exact tensor functor $\omega\colon\Rep_{R}(G)\rightarrow
\Mod_{R}$, the map $\alpha\mapsto\alpha(\xi_{0})$ identifies $\Hom^{\otimes
}(\omega^{G},\omega)$ with the set of isomorphisms $\omega^{G}(\xi
_{0})\rightarrow\omega(\xi_{0})$ mapping each $\omega(t_{i})$ to $t_{i}$. Here
$\omega^{G}$ denotes the forgetful functor on $\Rep_{R}(G)$.
\end{enumerate}
\end{proposition}

\begin{proof}
(a) If $\omega$ exists, then according to Proposition \ref{a1} there exists an
isomorphism $\omega^{G}\rightarrow\omega$, and so the condition is necessary.
For the converse, let $S$ be an $R$-algebra, and define $P(S)$ to be the set
of isomorphisms $S\otimes_{R}\Lambda(\xi_{0})\rightarrow S\otimes_{R}\Lambda$
mapping each $t_{i}$ to $s_{i}$. Then $S\rightsquigarrow P(S)$ is a $G$-torsor
which, by assumption, is trivial. The twist of $\omega^{G}$ by $P$ is an
$R$-valued fibre functor that satisfies (\ref{eq1}).

(b) Both sets are principal homogeneous space for $G(R)$, and so any
$G(R)$-equivariant map from one to the other is a bijection.
\end{proof}

\subsection{Objects with $G$-structure}

\begin{definition}
\label{a3}Let $R$ be a $Q$-algebra, and let $\mathsf{C}$ be an $R$-linear
rigid abelian tensor category. Let $G$ be a reductive group over $R$. An
\emph{object}\emph{ in }$\mathsf{C}$ \emph{with a }$G$\emph{-structure}%
\footnote{The concept was used in the original version of this article
(\cite{milne1995}). For this version, I have borrowed the name from
\cite{simpson1992}, p86, and \cite{rapoportR1996}, 3.3.} (or, more briefly, a
$G$\emph{-object} in $\mathsf{C}$) is an exact faithful tensor functor
$M\colon\Rep(G)\rightarrow\mathsf{C}$. We say that two $G$-objects are
\emph{equivalent} if they are equivalent as tensor functors.
\end{definition}

Let $(\xi_{0},(t_{i})_{i\in I})$ be defining for $G$. A $G$-object $M$ in
$\mathsf{C}$ defines an object $M(\xi_{0})$ together with a family of tensors
$(M(t_{i})_{i\in I})$. Loosely speaking, one can think of a $G$-object in
$\mathsf{C}$ as an object with a family of tensors satisfying some
condition.\footnote{One change from the first version of the article is that I
have chosen to work directly with $G$-objects rather than choosing a defining
representation and tensors. This avoids making choices and reveals the basic
constructions to be more obviously canonical.}

When the representations of $G$ can be described explicitly, so can the
$G$-objects in $\mathsf{C}$ (cf. \cite{rapoportR1996}, 3.3).

\begin{example}
\label{a4}If $G=\GL(V)$, then to give a $G$-object in $\mathsf{C}$ amounts to
giving an object $X$ in $\mathsf{C}$ of dimension $\dim V$. To see this, note
that because exact tensor functors preserve traces, a $G$-object
$M\colon\Rep(G)\rightarrow\mathsf{C}$ will map $V$ to an object $M(V)$ of
dimension $\dim V$ in $\mathsf{C}$. Conversely, let $X$ be an object of
dimension $\dim V$ in $\mathsf{C}$. For each $n\in\mathbb{N}$, choose an
object $\Sym^{n}V$ in $\Rep(G)$, and for each partition $\lambda$ of $n$, let
$S_{\lambda}V$ be the image of the Schur operator in $\Sym^{n}V$. Then the
representations $S_{\lambda}V$ form a set of representatives for the
isomorphism classes of simple representations of $G$ (\cite{fultonH1991},
15.47). There exists a skeleton $\Rep(G)^{\prime}$ whose objects are direct
sums of $S_{\lambda}V$s. Now repeat the process in $\mathsf{C}$ with $V$
replaced by $X$. There is an exact tensor functor $\Rep(G)^{\prime}%
\rightarrow\mathsf{C}$ sending each $S_{\lambda}V$ to $S_{\lambda}X$ and each
chosen direct sum in $\Rep(G)^{\prime}$ to the corresponding direct sum in
$\mathsf{C}$. The functor $\Rep(G)^{\prime}\rightarrow\Rep(G)$ is a tensor
equivalence, and so has a tensor inverse (\cite{saavedra1972}, I 4.4).
Therefore, we obtain a $G$-object $M$ of $\mathsf{C}$ with $M(V)=X$. Moreover,
$M$ is uniquely determined by $X$ up to a unique isomorphism.
\end{example}

\begin{example}
\label{a5}If $G=\SP_{n}$, then to give a $G$-object in $\mathsf{C}$ amounts to
giving an object $X$ of dimension $n$ together with a non-degenerate
alternating pairing%
\[
X\otimes X\rightarrow\1\text{.}%
\]
The proof is similar to the last example.
\end{example}

\subsubsection{Objects with $G$-structure in a Tate triple}

Recall that a \emph{Tate triple} \textsf{\thinspace}$\mathsf{T}$ over a field
$Q$ is a tannakian category $\mathsf{C}$ over $Q$ equipped with a (weight)
$\mathbb{Z}$-gradation $w\colon\mathbb{G}_{m}\mathrm{\rightarrow}%
\underline{\Aut}^{\otimes}(\id_{\mathsf{C}})$ and an invertible (Tate) object
$T$ of weight $-2$. A tensor functor of Tate triples is a tensor functor
$\eta$ of tannakian categories preserving the gradation together with an
isomorphism $\eta(T)\rightarrow T^{\prime}$. A fibre functor on $\mathsf{T}$
is a fibre functor $\omega$ on $\mathsf{C}$ together with an isomorphism
$\omega(T)\rightarrow\omega(T^{\otimes2})$ (equivalently, an isomorphism
$R\rightarrow\omega(T)$).

\begin{example}
\label{a6}The category $\mathrm{\Hdg}_{\mathbb{Q}}$ of rational Hodge
structures becomes a Tate triple with the weight gradation and the Tate object
$\mathbb{Q}(1)$ (equal to $2\pi i\mathbb{Q}\subset\mathbb{C}$ with the Hodge
structure of weight $-2$).
\end{example}

\begin{example}
\label{a7}To give a Tate triple structure on $\Rep(G)$ is the same as giving a
central homomorphism $w\colon\mathbb{G}_{m}\rightarrow G$ and a homomorphism
$t\colon G\rightarrow\mathbb{G}_{m}$ such that $t\circ w=-2$. The Tate object
is any one-dimensional space on which $G$ acting through $t$. We shall call
$(t,w)$ a \emph{Tate triple structure} on $G$.
\end{example}

Consider $(G,w,t)$ and a Tate triple $\mathsf{T}=(\mathsf{C},w,T)$. An object
in $\mathsf{T}$ with a $(G,w,t)$-structure (or, more briefly, a $(G,w,t)$%
-object in $\mathsf{T}$ or $\mathsf{C}$) is an exact tensor functor
$M\colon\Rep(G)\rightarrow\mathsf{T}$ of Tate triples.

\begin{example}
\label{a8}Let $\psi$ be a nondegenerate alternating pairing $V\times
V\rightarrow Q$ on the finite dimensional vector space $V$, and let
$G=\GSp(\psi)$. Thus%
\[
G(Q)=\{g\in\GL(V)\mid\psi(gv,gv^{\prime})=t(g)\psi(v,v^{\prime})\text{ some
}t(g)\in Q^{\times}\}.
\]
Let $w\colon\mathbb{G}_{m}\rightarrow G$ be the homomorphism such that $w(c)$
acts on $V$ as multiplication by $c^{-1}$ for $c\in Q^{\times}$. Then $(w,t)$
is a Tate triple structure on $G$, and to give a $(G,w,t)$-object in a Tate
triple $(\mathsf{C},w,T)$ amounts to giving an object $X$ of $\mathsf{C}$ of
dimension $\dim V$ together with a nondegenerate alternating pairing%
\[
X\otimes X\rightarrow T.
\]

\end{example}

\section{The category of motives over $\mathbb{F}$}

In this section, we define a category of motives $\Mot(\mathbb{F})$ over
$\mathbb{F}$ with the Weil protorus as its fundamental group, standard fibre
functors, and a canonical polarization; moreover, there is a reduction functor
from the category $\CM(\mathbb{Q}^{\text{\textrm{al}}})$ of CM-motives to
$\Mot(\mathbb{F}{})$. At present, the category depends on the choice of a
fibre functor $\omega_{0}$ on $\CM(\mathbb{Q}^{\text{\textrm{al}}})$ with
certain properties. However, the rationality conjecture for CM abelian
varieties (\cite{milne2009}, 4.1) implies that there is a unique preferred
$\omega_{0}$, and the Hodge conjecture for abelian varieties of CM-type
implies that, with this $\omega_{0}$, $\Mot(\mathbb{F})$ is indeed the
category of abelian motives over $\mathbb{F}$ defined using algebraic cycles
modulo numerical equivalence.

\subsection{The realization categories}

\subsubsection{The realization category at infinity.}

Let $\mathsf{R}_{\infty}$ be the category of pairs $(V,F)$ consisting of a
$\mathbb{Z}$-graded finite-dimensional complex vector space $V=\bigoplus
\nolimits_{m\in\mathbb{Z}}V^{m}$ and an $\iota$-semilinear endomorphism $F$
such that $F^{2}=(-1)^{m}$ on $V^{m}$. With the obvious tensor structure,
$\mathsf{R}_{\infty}$ becomes a tannakian category over $\mathbb{R}$ with
fundamental group $\mathbb{G}_{m}$. The objects fixed by $\mathbb{G}_{m}$ are
those of weight zero. If $(V,F)$ is of weight zero, then%
\[
V^{F}\overset{\textup{{\tiny def}}}{=}\{v\in V\mid Fv=v\}
\]
is an $\mathbb{R}$-structure on $V$. The functor $V\rightsquigarrow V^{F}$ is
an $\mathbb{R}$-valued fibre functor on $\mathsf{R}_{\infty}^{\mathbb{G}_{m}}$.

\subsubsection{The realization category at $\ell\neq p,\infty.$}

Let $\mathsf{R}_{\ell}$ be the category of finite-dimensional $\mathbb{Q}%
_{\ell}$-vector spaces. It is a tannakian category with trivial fundamental
group. The forgetful functor is a $\mathbb{Q}_{\ell}$-valued fibre functor on
$\mathsf{R}_{\ell}$.

\subsubsection{The crystalline realization category.}

Let $\mathsf{R}_{p}$ be the category of $F$-isocrystals over $\mathbb{F}$.
Thus, an object of $\mathsf{R}_{p}$ is a pair $(V,F)$ consisting of a
finite-dimensional vector space $V$ over $B(\mathbb{F})$ and a $\sigma
$-semilinear isomorphism $F\colon V\rightarrow V$. With the obvious tensor
structure, $\mathsf{R}_{p}$ becomes a tannakian category over $\mathbb{Q}_{p}$
whose fundamental group is the universal covering group $\mathbb{G}$ of
$\mathbb{G}_{m}$ (so $X^{\ast}(\mathbb{G})=\mathbb{Q}$). The objects fixed by
$\mathbb{G}$ are those of slope zero. If $(V,F)$ is of slope zero, then%
\[
V^{F}\overset{\textup{{\tiny def}}}{=}\{v\in V\mid Fv=v\}
\]
is a $\mathbb{Q}_{p}$-structure on $V$. The functor $V\rightsquigarrow V^{F}$
is a $\mathbb{Q}_{p}$-valued fibre functor on $\mathsf{R}_{p}^{\mathbb{G}}$.

\subsection{The category of CM-motives over $\mathbb{Q}^{\mathrm{al}}$}

By a Hodge class on an abelian variety $A$ over a field $k$ of characteristic
zero, we mean an absolute Hodge cycle in the sense of \cite{deligne1982}, and
we let $\mathcal{B}(A)$ denote the $\mathbb{Q}$-algebra of such classes on $A$.

Let $\CM(\mathbb{Q}^{\mathrm{al}})$ be the category of CM-motives over
$\mathbb{Q}^{\mathrm{al}}$. Thus, an object of $\CM(\mathbb{Q}^{\text{al}})$
is a triple $(A,e,m)$ with $A$ an abelian variety of CM-type over
$\mathbb{Q}^{\mathrm{al}}$, $e$ an idempotent in the ring $\mathcal{B}^{\dim
A}(A\times A)$, and $m$ an integer; the morphisms are given by the rule,%
\[
\Hom((A,e,m),(B,f,n))=f\cdot\mathcal{B}^{\dim A-m+n}(A\times B)\cdot e.
\]
With the usual tensor structure, $\CM(\mathbb{Q}^{\mathrm{al}})$ becomes a
semisimple tannakian category over $\mathbb{Q}$. Its fundamental group is the
Serre group $S$. Recall that $S$ has a canonical (weight) cocharacter
$w=w^{S}$ defined over $\mathbb{Q}$, and a canonical cocharacter $\mu=\mu^{S}$
such that $w=-(1+\iota)\mu$; moreover, the pair $(S,\mu^{S})$ is universal.

\subsubsection{The local realization at $\infty$.}

Let $(V,h)$ be a real Hodge structure, and let $C$ act on $V$ as $h(i)$. Then
the square of the operator $v\mapsto C\bar{v}$ acts as $(-1)^{m}$ on $V^{m}$.
Therefore, $\mathbb{C}\otimes_{\mathbb{R}}V$ endowed with its weight gradation
and this operator is an object of $\mathsf{R}_{\infty}.$ We let
\[
\xi_{\infty}\colon\CM(\mathbb{Q}^{\mathrm{al}})\rightarrow\mathsf{R}_{\infty
},\quad X\rightsquigarrow(\omega_{B}(X)_{\mathbb{R}{}},C),
\]
denote the functor sending $X$ to the object of $\mathsf{R}_{\infty}$ defined
by the real Hodge structure $\omega_{B}(X)_{\mathbb{R}}$. Then $\xi_{\infty}$
is an exact tensor functor, and the cocharacter $x_{\infty}\colon
\mathbb{G}_{m}\rightarrow S_{\mathbb{R}}$ it defines is equal to
$w_{\mathrm{\mathbb{R}}}$. We obtain an $\mathbb{R}$-valued fibre functor
$\omega_{\infty}$ on $\CM(\mathbb{Q}^{\mathrm{al}})^{\mathbb{G}_{m}}$ as follows:%

\[
\xymatrixcolsep{3pc}\xymatrix{
\CM(\mathbb{Q}^{\mathrm{al}})^{\mathbb{G}_m}\ar[r]_-{\xi_{\infty}}\ar@/^1pc/[rr]^{\omega_\infty}
&\mathsf{R}_{\infty}^{\mathbb{G}_m}\ar[r]_-{V\rightsquigarrow
V^{F}}&\Vc(\mathbb{R}).
}
\]

\subsubsection{The local realization at $\ell$.}

For each $\ell\neq p,\infty$, we let $\omega_{\ell}$ denote the fibre functor
on $\CM(\mathbb{Q}^{\mathrm{al}})$ defined by $\ell$-adic \'{e}tale cohomology.

\subsubsection{The local realization at $p$.}

A CM abelian variety $A$ over $\mathbb{Q}^{\mathrm{al}}$ has good reduction at
the prime $v$ to an abelian variety $A_{0}$ over $\mathbb{F}$. The map%
\[
(A,e,m)\mapsto e\cdot H_{\mathrm{crys}}^{\ast}(A_{0})(m)
\]
extends to an exact tensor functor
\[
\xi_{p}\colon\CM(\mathbb{Q}^{\mathrm{al}})\rightarrow\mathsf{R}_{p}\text{.}%
\]
Let $x_{p}$ denote the homomorphism $\mathbb{G}\rightarrow S_{\mathbb{Q}_{p}}$
defined by $\xi_{p}$. We obtain a $\mathbb{Q}_{p}$-valued fibre functor
$\omega_{p}$ on $\CM(\mathbb{Q}^{\mathrm{al}})^{\mathbb{G}}$ as follows:%
\[
\xymatrixcolsep{3pc}\xymatrix{
\CM(\mathbb{Q}^{\mathrm{al}})^{\mathbb{G}}\ar[r]_-{\xi_{p}}\ar@/^1pc/[rr]^{\omega_p}
&\mathsf{R}_{p}^{\mathbb{G}}\ar[r]_-{V\rightsquigarrow
V^{F}}&\Vc(\mathbb{Q}_{p}).
}
\]

\subsection{The Shimura-Taniyama homomorphism}

Let $P$ be the Weil-number protorus (see, for example, \cite{milne1994},
\S 2). Thus, $P$ is a protorus over $\mathbb{Q}$, and every element of
$W\overset{\textup{{\tiny def}}}{=}X^{\ast}(P)$ is represented by a Weil
$p^{n}$-number $\pi$; two pairs $(\pi,n)$ and $(\pi^{\prime},n^{\prime})$
represent the same element of $P$ if and only if $\pi^{n^{\prime}N}=\pi^{nN}$
for some integer $N\geq1$. Define homomorphisms%
\begin{align*}
z_{\infty}\colon\mathbb{G}_{m}  &  \rightarrow P,\quad\langle\lbrack
\pi,n],z_{\infty}\rangle=m\text{ if }|\pi|_{\infty}=(p^{n/2})^{m},\\
z_{p}\colon\mathbb{G}  &  \rightarrow P,\quad\langle\lbrack\pi,n],z_{p}%
\rangle=-\frac{\ord_{v}(\pi)}{\ord_{v}(p^{n})}\text{.}%
\end{align*}
There is a unique homomorphism $r^{\mathrm{ST}}\colon P\rightarrow S$, which I
call the \emph{Shimura-Taniyama homomorphism,} sending $z_{\infty}$ to
$x_{\infty}$ and $z_{p}$ to $x_{p}$. The homomorphism $r^{\mathrm{ST}}$ is
injective, which allows us to identify $P$ with a subgroup of $S$.

\begin{aside}
Let $K$ be a CM field. A CM-type on $K$ is a cocharacter $f$ of $S^{K}$ of
weight $-1$ taking values in $\{0,1\}$. Let $A_{f}$ be an abelian variety of
CM-type $f$ over $\mathbb{Q}^{\mathrm{al}}$. Then $A_{f}$ has good reduction
to an abelian variety $B_{f}$ over $\mathbb{F}$, which defines an element
$[\pi(f)]$ of $W^{K}$. The Shimura-Taniyama formula expresses $[\pi(f)]$ in
terms of $f$. The Shimura-Taniyama homomorphism is the unique homomorphism
$r\colon P\rightarrow S$ such that $X^{\ast}(r)$ maps every CM-type $f$ to
$[\pi(f)]$. This explains the choice of name.
\end{aside}

\subsection{The category of abelian motives over $\mathbb{F}$}

For a $\mathbb{Q}$-valued fibre functor $\omega$ on a tannakian category, let
$\omega(l)$ denote the $\mathbb{Q}_{l}$-valued fibre functor
$X\rightsquigarrow\mathbb{Q}_{l}\otimes_{\mathbb{Q}}\omega(X)$.

The fibre functors $\omega_{l}$ on $\CM(\mathbb{Q}^{\mathrm{al}})$ constructed
above restrict to fibre functors $\omega_{l}|$ on $\CM(\mathbb{Q}%
^{\mathrm{al}})^{P}$.

\begin{theorem}
\label{a9}There exists a $\mathbb{Q}$-valued fibre functor $\omega_{0}$ on
$\CM(\mathbb{Q}^{\mathrm{al}})^{P}$ such that%
\[
\omega_{0}(l)\approx\omega_{l}|\quad
\]
for all $l$ (including $p$ and $\infty$).
\end{theorem}

\begin{proof}
Let $\omega_{B}$ be the Betti fibre functor on $\CM(\mathbb{Q}%
^{\text{\textrm{al}}})$. For a $\mathbb{Q}$-valued fibre functor $\omega$ on
$\CM(\mathbb{Q}^{\mathrm{al}})^{P}$, let $\wp(\omega)=\underline
{\Hom}^{\otimes}(\omega_{B},\omega)$. This is an $S/P$-torsor, and the theory
of tannakian categories shows that every $S/P$-torsor is isomorphic to
$\wp(\omega)$ for some fibre functor $\omega$. The fibre functor $\omega$ will
satisfy the condition in the proposition if and only if the class of
$\wp(\omega)$ in $H^{1}(\mathbb{Q},S/P)$ maps to the class of $\wp(\omega
_{l})$ in $H^{1}(\mathbb{Q}_{l},S/P)$ for all $l$. That such a class exists is
proved in \cite{milne2003}, Theorem 4.1.\footnote{It should be noted that the
proof of \cite{milne2003}, Lemma 4.2, requires the main result of
\cite{wintenberger1991} in order to replace the cohomology class $c_{p}^{K}$
with an elementarily defined class (cf.\ ibid. p31).}
\end{proof}

The isomorphism class of the restriction of $\omega_{0}$ to any algebraic
subcategory of $\CM(\mathbb{Q}^{\mathrm{al}})^{P}$ is uniquely determined, but
this is not true for $\omega_{0}$ itself (\cite{milne2003}, 1.11). The
rationality conjecture (\cite{milne2009}, 4.1) for CM abelian varieties
implies that there is a unique preferred $\omega_{0}$ satisfying the condition
of the proposition.\footnote{In more detail, the rationality conjecture for CM
abelian varieties implies that there exists a (unique!) good theory of
rational Tate classes on abelian varieties over finite fields and that Hodge
classes on CM abelian varieties reduce to rational Tate classes. We can use
the rational Tate classes to construct a category $\Mot(\mathbb{F}{})$ of
abelian motives over $\mathbb{F}{}$, and there will be a canonical reduction
functor $R\colon\CM(\mathbb{Q}^{\mathrm{al}})\rightarrow\Mot(\mathbb{F}{})$.
The functor $X\rightsquigarrow\Hom(\1,R(X))$ is a $\mathbb{Q}{}$-valued fibre
functor $\omega_{0}$ on $\CM(\mathbb{Q}^{\mathrm{al}})^{P}$, and $R$ defines
an equivalence $\CM(\mathbb{Q}^{\mathrm{al}})/\omega_{0}\rightarrow
\Mot(\mathbb{F)}$. If the rationality conjecture holds for all abelian
varieties over $\mathbb{Q}{}^{\mathrm{al}}$ with good reduction, then this
gives a functor from the category of motives generated by such abelian
varieties to $\Mot(\mathbb{F)}$, which would simplify the proof of the
Langlands-Rapoport conjecture.}

Choose a fibre functor $\omega_{0}$ as in the proposition, and define
$\Mot_{\omega_{0}}(\mathbb{F})$ to be the corresponding quotient category
\[
R\colon\CM(\mathbb{Q}^{\mathrm{al}})\rightarrow\Mot_{\omega_{0}}(\mathbb{F})
\]
in the sense of \cite{milne2007}. Thus $\Mot_{\omega_{0}}(\mathbb{F})$ is a
semisimple tannakian category over $\mathbb{Q}$ with fundamental group $P,$
which can be described as follows. For a CM abelian variety $A$ over
$\mathbb{Q}^{\mathrm{al}}$, let $\mathcal{B}_{\omega_{0}}(A)=\omega
_{0}(h(A)^{P})$ where $h(A)$ is the object $(A,1,0)$ of $\CM(\mathbb{Q}%
^{\text{al}})$. The objects of $\Mot_{\omega_{0}}(\mathbb{F})$ are the triples
$(A,e,m)$ with $A$ a CM abelian variety over $\mathbb{Q}^{\mathrm{al}}$, $e$
an idempotent in the ring $\mathcal{B}_{\omega_{0}}^{\dim A}(A\times A)$, and
$m\in\mathbb{Z}$. We sometimes write $\hbar(A,e,m)$ for the object $(A,e,m)$
of $\Mot_{\omega_{0}}(\mathbb{F})$. Note that the maps%
\[
\mathcal{B}(A)\simeq\Hom(\1,h\left(  A\right)  )\simeq\Hom(\1,h\left(
A\right)  ^{P})\overset{\omega_{0}}{\hookrightarrow}\Hom(\mathbb{Q}%
,\mathcal{B}_{\omega_{0}}(A))\simeq\mathcal{B}_{\omega_{0}}(A)
\]
realize $\mathcal{B}(A)$ as a subspace of $\mathcal{B}_{\omega_{0}}(A)$. The
functor $R$ is
\[
(A,e,m)\rightsquigarrow\hbar(A,e,m)
\]
(on the right $e$ is to be regarded as an element of $\mathcal{B}_{\omega_{0}%
}(A)$). For any objects $X$ and $Y$ of $\CM(\mathbb{Q}^{\mathrm{al}})$,%
\[
\Hom(RX,RY)=\omega_{0}(\underline{\Hom}(X,Y)^{P}).
\]

\noindent\begin{minipage}{3in}
The choice of an isomorphism $\omega_{0}(l)\rightarrow\omega_{l}|$ determines
an exact tensor functor
\[\zeta_{l}\colon\Mot_{\omega_{0}}(\mathbb{F})\rightarrow\mathsf{R}_{l}\]
such $\zeta_{l}\circ R=\xi_{l}$ (see
\cite{milne2007}, end of \S 2). Therefore, such a choice determines a
commutative diagram of tannakian categories and exact tensor functors as at right:
\end{minipage}\hfill\begin{minipage}{2.5in}\[
\xymatrixrowsep{1pc}
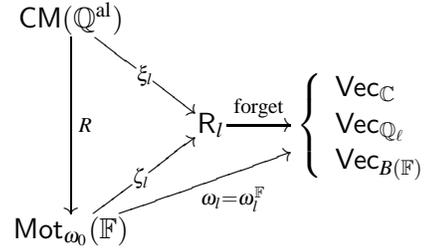
\xymatrix@1{\CM(\mathbb{Q}^{\text{al}})\ar[dd]^{R}\ar[rd]|-{\xi_{l}}\\
&\mathsf{R}_{l}\ar[r]^-{\text{forget}}&{\left\{
\begin{array}
[c]{l}%
\Vc_{\mathbb{C}}\\
\Vc_{\mathbb{Q}_{\ell}}\\
\Vc_{B(\mathbb{F})}
\end{array}
\right.  }\\
\Mot_{\omega_0}(\mathbb{F})\ar[ru]|{\zeta_{l}}\ar[rru]_{\omega_{l}=\omega_{l}^\mathbb{F}}
}\]
\end{minipage}

\noindent The canonical polarization on $\CM(\mathbb{Q}^{\text{al}})$
(regarded as a Tate triple) passes to the quotient and defines a polarization
on $\Mot(\mathbb{F}{})$ (see \cite{milne2002}).

Note that, for a motive $M$ over $\mathbb{F}$, we have defined $\omega
_{l}^{\mathbb{F}}(M)$ to be the vector space underlying $\zeta_{l}(M)$, and so
for $l=p,\infty$ it has a Frobenius operator $F$ and for $l\neq p,\infty$ it
has a germ of a Frobenius operator (\cite{milne1994}, p422).

\noindent\emph{We now fix a fibre functor }$\omega_{0}$\emph{ and isomorphisms
}$\omega_{0}(l)\rightarrow\omega_{l}|$ \emph{as above, and we write
}$\Mot(\mathbb{F})$ for $\Mot_{\omega_{0}}(\mathbb{F})$.

\subsection{The category of motives over $\mathbb{F}$ with $\mathbb{Z}{}%
_{(p)}$-coefficients}

Define a \emph{motive} $M$ \emph{over} $\mathbb{F}$ \emph{with coefficients in
}$\mathbb{Z}{}_{(p)}$ to be a triple $(M_{p},M_{0},m)$ consisting of

\begin{enumerate}
\item a finitely generated $W(\mathbb{F})$-module $M_{p}$ and a $\sigma
$-linear map $F\colon M_{p}\rightarrow M_{p}$ whose kernel is torsion,

\item an object $M_{0}$ of $\Mot(\mathbb{F})$, and

\item an isomorphism $m\colon(M_{p})_{\mathbb{Q}}\rightarrow\omega
_{p}^{\mathbb{F}}(M_{0})$.
\end{enumerate}

\noindent A morphism of motives $\alpha\colon M\rightarrow N$ with
coefficients in $\mathbb{Z}{}_{(p)}$ is a pair of morphisms%
\[
(\alpha_{p}\colon M_{p}\rightarrow N_{p},\alpha_{0}\colon M_{0}\rightarrow
N_{0})
\]
such that $n\circ\alpha_{p}=\omega_{p}^{\mathbb{F}}(\alpha_{0})\circ m$. Let
$\Mot_{(p)}(\mathbb{F})$ be the category of motives over $\mathbb{F}{}$ with
coefficients in $\mathbb{Z}{}_{(p)}$, and let $\Mot_{(p)}^{\prime}%
(\mathbb{F})$ be the full subcategory of objects $M$ such that $M_{p}$ is
torsion-free. The objects of $\Mot_{(p)}^{\prime}(\mathbb{F})$ will be called torsion-free.

\begin{proposition}
\label{a10}With the obvious tensor structure, $\Mot_{(p)}(\mathbb{F})$ is an
abelian tensor category over $\mathbb{Z}_{(p)}$, and $\Mot_{(p)}^{\prime
}(\mathbb{F})$ is a rigid pseudo-abelian tensor subcategory of $\Mot_{(p)}%
(\mathbb{F})$. Moreover:

\begin{enumerate}
\item the tensor functor $M\rightsquigarrow M_{p}\colon\Mot_{(p)}%
(\mathbb{F})\rightarrow\Mod_{W(\mathbb{F})}$ is exact;

\item there are canonical equivalences of categories
\[
\Mot_{(p)}^{\prime}(\mathbb{F})_{(\mathbb{Q})}\rightarrow\Mot_{(p)}%
(\mathbb{F})_{(\mathbb{Q})}\rightarrow\Mot(\mathbb{F});
\]

\item for all torsion-free motives $M,N$, the cokernel of%
\[
\Hom(M,N)\otimes_{\mathbb{Z}_{(p)}}W(\mathbb{F})\rightarrow\Hom(M_{p},N_{p})
\]
is torsion-free.
\end{enumerate}
\end{proposition}

\begin{proof}
These can be proved in the same way as the similar statements in
\cite{milneR2004}.
\end{proof}

\section{Conjecture LR+}

In this section we state a conjecture that both strengthens and simplifies the
original conjecture of Langlands and Rapoport.

Throughout this section, $(G,X)$ is a Shimura $(p)$-datum satisfying (SV1-4,6)
--- see p\pageref{pdatum}.

\subsection{Motives with $G$-structure}

By a motive over $\mathbb{F}$ with $G$\emph{-structure}, or, more briefly, a
$G$\emph{-motive} over $\mathbb{F}$, we mean an exact tensor functor
$M\colon\Rep_{\mathbb{Q}}(G)\rightarrow\Mot(\mathbb{F})$. We say that two
$G$-motives are \emph{equivalent} if they are isomorphic as tensor functors.

We shall be especially interested in the $G$-motives for which%
\begin{equation}
\omega_{\infty}^{\mathbb{F}}\circ M\approx\mathbb{C}\otimes_{\mathbb{Q}%
}V,\quad\omega_{\ell}^{\mathbb{F}}\circ M\approx\mathbb{Q}_{\ell}%
\otimes_{\mathbb{Q}}V,\quad\omega_{p}^{\mathbb{F}}\circ M\approx
B(\mathbb{F})\otimes_{\mathbb{Q}}V,\quad\label{eq4}%
\end{equation}
(isomorphisms of fibre functors on $\Rep(G)$). Here $V$ denotes the forgetful
fibre functor $\xi\rightsquigarrow V(\xi)$ on $\Rep_{\mathbb{Q}}(G)$, and
$R\otimes_{\mathbb{Q}}V$ denotes the fibre functor $\xi\rightsquigarrow
R\otimes_{\mathbb{Q}}V(\xi)$.

\subsection{Integral structures at $p\qquad$}

\begin{plain}
\label{a11}Let $(G,X)$ be a Shimura $p$-datum. Recall that the reflex field
$E=E(G,X)$ is the field of definition of the $G(\mathbb{C})$-conjugacy class
$\mathcal{C}$ of cocharacters of $G_{\mathbb{C}}$ containing $\mu_{x}$ for all
$x\in X$. Because $G_{\mathbb{Q}_{p}}$ admits a hyperspecial subgroup, the
prime $v$ in $E$ is unramified, and so the closure $E_{v}$ of $E$ in
$(\mathbb{Q}^{\mathrm{al}})_{v}$ is a subfield of $B(\mathbb{F})$. Let $S$ be
a maximal split subtorus of $G_{B(\mathbb{F)}}$ whose apartment in the
Bruhat-Tits building of $G_{B(\mathbb{F})}$ contains the hyperspecial vertex
stabilized by $G(W(\mathbb{F}))$. Then $\mathcal{C}$ is represented by a
cocharacter $\mu_{0}$ of $S$ defined over $B(\mathbb{F})$ whose orbit under
the action of the Weyl group of $S$ is uniquely determined; moreover, the
elements of the Weyl group of $S$ are represented by elements in
$G(W(\mathbb{F}))$. (For more details and references, see \cite{milne1994s}, pp503--4.)
\end{plain}

Recall that $\Rep_{\mathbb{Q}}(G)_{(\mathbb{Q}_{l})}\simeq\Rep_{\mathbb{Q}%
_{l}}(G)$. Therefore, a $G$-motive $M$ defines by extension of scalars a
functor
\[
M(p)\colon\Rep_{\mathbb{Q}_{l}}(G)\rightarrow\Mot(\mathbb{F})_{(\mathbb{Q}%
_{l})}.
\]
We sometimes write $M_{(l)}$ for $\omega_{l}^{\mathbb{F}{}}\circ M$ (or its
extension $\omega_{l}^{\mathbb{F}{}}\circ M(l)$ to $\Rep_{\mathbb{Q}{}_{l}%
}(G)$).

Let $M$ be a $G$-motive such that $M_{(p)}\approx B(\mathbb{F})\otimes
_{\mathbb{Q}}V$.

\begin{definition}
\label{a12}A $p$\emph{-integral structure\/} on $M$ is an exact tensor functor
$\underline{\Lambda}\colon\Rep_{\mathbb{Z}_{p}}(G)\rightarrow
\Mod_{W(\mathbb{F})}$ such that

\begin{enumerate}
\item for all $\xi$ in $\Rep_{\mathbb{Z}_{p}}(G)$, we have that $\underline
{\Lambda}(\xi)$ is a $W(\mathbb{F})$-lattice in $M_{(p)}(\xi_{\mathbb{Q}_{p}%
})$, and

\item there exists an isomorphism $\eta\colon W\otimes_{\mathbb{Z}_{p}}%
\Lambda\rightarrow\underline{\Lambda}$ of tensor functors such that
$\eta_{B(\mathbb{F})}$ maps $\mu_{0}(p^{-1})\cdot\Lambda(\xi)_{W}$ onto
$F\underline{\Lambda}(\xi)$ for all $\xi$ in $\Rep_{\mathbb{Z}_{p}}(G)$.
\end{enumerate}
\end{definition}

\noindent In (b), $\Lambda$ denotes the forgetful functor $(\Lambda
,\xi)\rightsquigarrow\Lambda$ on $\Rep_{\mathbb{Z}_{p}}(G)$.

\begin{remark}
\label{a13}(a) Condition (a) means that $\underline{\Lambda}(\xi)\subset
M_{(p)}(\xi_{\mathbb{Q}_{p}})$ for all $\xi$, and that there is a commutative
diagram%
\[
\xymatrix{
\Rep_{\mathbb{Q}_{p}}(G)\ar[r]^{M(p)}
&\Mot(\mathbb{F})_{\left(  \mathbb{Q}_{p}\right)}\ar[r]^{\omega^{\mathbb{F}}_{p}}
&\Vc_{B(\mathbb{F})}\\
\Rep_{\mathbb{Z}_{p}}(G)\ar[u]_{\mathbb{Q}_p\otimes_{\mathbb{Z}_p}-}\ar[rr]^{\underline{\Lambda}}
&&\Mod_{W(\mathbb{F})}\ar[u]_{B\otimes_{W}-}.}
\]
(b) We define a \emph{filtered} $B(\mathbb{F})$\emph{-module\/} to be an
$F$-isocrystal $(N,F)$ over $\mathbb{F}$ together with a finite filtration
\[
N=\Filt^{i_{0}}(N)\supset\cdots\supset\Filt^{i}(N)\supset\Filt^{i+1}%
(N)\supset\cdots\supset\Filt^{i_{1}}(N)=0
\]
on $N$. In Fontaine's terminology, a $W$-lattice $\Lambda$ in $N$ is
\emph{strongly divisible\/} if
\[
\sum\nolimits_{i}p^{-i}F(\Filt^{i}N\cap\Lambda)=\Lambda,
\]
and a filtered $B(\mathbb{F})$-module admitting a strongly divisible lattice
is said to be \emph{weakly admissible }(\cite{fontaine1983}, p90)\emph{.\/} If
$\mu\colon\mathbb{G}_{m}\rightarrow\GL(\Lambda)$ splits the filtration on
$\Lambda$, i.e.,
\[
\Filt^{j}\Lambda=\bigoplus\nolimits_{i\geq j}\Lambda^{i},\quad\Lambda
^{i}\overset{\textup{{\tiny def}}}{=}\{m\in\Lambda\mid\mu(x)m=x^{i}%
m\text{\textrm{, all }}x\in B(\mathbb{F})^{\times}\},
\]
then the condition to be strongly divisible is that $F\Lambda=\mu(p)\Lambda$.
The cocharacter $\mu_{0}^{-1}$ of $G_{B(\mathbb{F})}$ constructed in
(\ref{a11}) defines a filtration on $\mathbb{Q}\otimes\Lambda(\xi)$ for all
$\xi$, and $\mu_{0}$ has been chosen so that $\mu_{0}^{-1}$ splits the
filtration on $\Lambda(\xi)$ for all $\xi$. Thus condition (b) for
$\underline{\Lambda}$ to be a $p$-integral structure on $M$ can be restated
as:\bquote there exists an isomorphism $\eta\colon W\otimes_{\mathbb{Z}_{p}%
}\Lambda\rightarrow\underline{\Lambda}$ of tensor functors such that, for all
$\xi\in\Rep_{\mathbb{Z}_{p}}(G)$, $\underline{\Lambda}(\xi)$ is strongly
divisible for the filtration on $B(\mathbb{F})\otimes_{W}\underline{\Lambda
}(\xi)$ defined (via $\eta$) by $\mu_{0}^{-1}$.\equote\ 
\end{remark}

\subsection{The set $\mathcal{L}(M)$}

Let $M$ be a $G$-motive satisfying the condition (\ref{eq4}), p\pageref{eq4}.
For a $p$-structure $\underline{\Lambda}$ on $M$, define $\Phi\underline
{\Lambda}$ to be the $p$-structure such that%
\[
(\Phi\underline{\Lambda})(\xi)=F^{[k(v):\mathbb{F}_{p}]}\cdot\underline
{\Lambda}(\xi),\quad\text{for all }\xi\in\Rep_{\mathbb{Z}_{p}}(G).
\]
Here $k(v)$ is the residue field at the prime $v$ of $E$ (so $[k(v)\colon
\mathbb{F}_{p}]=[E_{v}\colon\mathbb{Q}_{p}]$).

Define
\begin{align*}
I(M)  &  =\Aut^{\otimes}(M),\\
X^{p}(M)  &  =\Isom^{\otimes}(\mathbb{A}_{f}^{p}\otimes_{\mathbb{Q}}%
V,\omega_{f}^{p}\circ M),\\
X_{p}(M)  &  =\{p\text{-integral structures on }M\}.
\end{align*}
The group $I(M)$ acts on both $X^{p}(M)$ and $X_{p}(M)$ on the left, and so we
can define
\[
\mathcal{L}(M)=I(M)\left\backslash \left(  X^{p}(M)/Z^{p}\right)  \times
X_{p}(M)\right.  .
\]
Here $Z^{p}$ is the closure of $Z(\mathbb{Z}_{(p)})$ in $Z(\mathbb{A}_{f}%
^{p})$ where $Z=Z(G)$. The group $G(\mathbb{A}_{f}^{p})$ acts on $X^{p}(M)$
through its action on $\mathbb{A}_{f}^{p}\otimes_{\mathbb{Q}}V$, and we let it
act on $\mathcal{L}(M)$ through its action on $X^{p}(M)$. We let
$Z(\mathbb{Q}_{p})$ and $\Phi$ act $\mathcal{L}(M)$ through their actions on
$X_{p}(M)$.

An isomorphism $M\rightarrow M^{\prime}$ of $G$-motives defines an equivariant
bijection $a_{M^{\prime},M}\colon\mathcal{\mathcal{L}}(M)\rightarrow
\mathcal{\mathcal{L}}(M^{\prime})$ which is independent of the choice of the
isomorphism. For an equivalence class $m$ of $G$-motives, we define
\[
\mathcal{\mathcal{L}}(m)=\varprojlim_{M\in m}\mathcal{\mathcal{L}}(M).
\]
This is a set with actions of $G(\mathbb{A}_{f}^{p})\times Z(\mathbb{Q}_{p})$
and of a Frobenius operator $\Phi$; it is equipped with equivariant
isomorphisms $a_{M}\colon\mathcal{L}(m)\rightarrow\mathcal{L}(M)$ such that
$a_{M^{\prime},M}\circ a_{M}=a_{M^{\prime}}$ for all $M$, $M^{\prime}\in m$.

\subsection{Special $G$-motives}

Recall that $\Hdg_{\mathbb{Q}}$ is the category of polarizable Hodge
structures over $\mathbb{Q}$. A point $x$ of $X$ defines an exact tensor
functor
\[
H_{x}\colon\Rep_{\mathbb{Q}}(G)\rightarrow\Hdg_{\mathbb{Q}},\quad
\xi\rightsquigarrow(V(\xi),\xi_{\mathbb{R}}\circ h_{x}).
\]
When $x$ is special, $H_{x}$ takes values in the full subcategory of
$\Hdg_{\mathbb{Q}}$ whose objects are the rational Hodge structures of CM-type
(because we are assuming (SV4) and (SV6) --- see the notations). This
subcategory is equivalent (via $\omega_{B}$) to $\CM(\mathbb{Q}^{\mathrm{al}%
})$. Fix a tensor inverse $\Hdg_{\mathbb{Q}}\rightarrow\CM(\mathbb{Q}%
^{\mathrm{al}})$ to $\omega_{B}$. On composing $H_{x}$ with it, we obtain a
tensor functor $M_{x}\colon\Rep_{\mathbb{Q}}(G)\rightarrow\CM(\mathbb{Q}%
^{\mathrm{al}})$ together with an isomorphism $\omega_{B}\circ M_{x}\simeq
H_{x}$. Any $G$-motive equivalent to $R\circ M_{x}$ for some special $x\in X$
will be called \emph{special. }

\begin{lemma}
\label{a14} Every special $G$-motive $M\colon\Rep_{\mathbb{Q}}(G)\rightarrow
\Mot(\mathbb{F})$ satisfies the condition (\ref{eq4}), p\pageref{eq4}.
\end{lemma}

\begin{proof}
We consider only the case $l=p$ since the other cases are easier. As the
statement depends only on the isomorphism class of $M$, we may assume that
$M=R\circ M_{x}$ with $x$ a special point of $X$.

The reduction functor $R\colon\CM(\mathbb{Q}^{\mathrm{al}})\rightarrow
\Mot(\mathbb{F})$ has the property that
\[
(\mathbb{Q}^{\mathrm{al}})_{v}\otimes_{B(\mathbb{F})}(\omega_{p}^{\mathbb{F}%
}\circ R)\simeq(\mathbb{Q}^{\mathrm{al}})_{v}\otimes_{\mathbb{Q}^{\mathrm{al}%
}}\omega_{\mathrm{dR}}%
\]
where $\omega_{\mathrm{dR}}$ denotes the de Rham fibre functor on
$\CM(\mathbb{Q}^{\text{al}})$. On composing both sides with $M_{x}$, we find
that
\[
(\mathbb{Q}^{\mathrm{al}})_{v}\otimes_{B(\mathbb{F})}(\omega_{p}^{\mathbb{F}%
}\circ M)\simeq(\mathbb{Q}^{\mathrm{al}})_{v}\otimes_{\mathbb{Q}^{\mathrm{al}%
}}(\omega_{\mathrm{dR}}\circ M_{x}).
\]
There is a comparison isomorphism
\[
\omega_{\mathrm{dR}}\rightarrow\mathbb{Q}^{\mathrm{al}}\otimes_{\mathbb{Q}%
}\omega_{B},
\]
and the definition of $M_{x}$ gives an isomorphism
\[
\omega_{B}\circ M_{x}\rightarrow V.
\]
On combining these isomorphisms, we obtain an isomorphism of tensor functors
\[
(\mathbb{Q}^{\mathrm{al}})_{v}\otimes_{B(\mathbb{F})}M_{(p)}\rightarrow
(\mathbb{Q}^{\mathrm{al}})_{v}\otimes_{\mathbb{Q}}V.
\]
It remains to show that we can replace $(\mathbb{Q}^{\mathrm{al}})_{v}$ with
$B(\mathbb{F})$ in this statement.

Consider the functor of $B(\mathbb{F})$-algebras
\[
R\rightsquigarrow\underline{\Isom}^{\otimes}(R\otimes_{\mathbb{Q}}%
V,R\otimes_{B(\mathbb{F})}M_{(p)}).
\]
This is a pseudo-torsor for $\underline{\Aut}^{\otimes}(B(\mathbb{F}%
)\otimes_{\mathbb{Q}}V)=G_{B(\mathbb{F})}$ and, in fact, a torsor because it
has a $(\mathbb{Q}^{\mathrm{al}})_{v}$-point. It therefore defines an element
of $H^{1}(B(\mathbb{F}),G)$. As the field $B(\mathbb{F})$ has dimension
$\leq1$ and $G$ is connected, $H^{1}(B(\mathbb{F}),G)=0$ (\cite{steinberg1965}%
, 1.9), and so the torsor is trivial.
\end{proof}

In particular, when $M$ is special, the set $\mathcal{L}(M)$ is defined.

\subsection{Shimura varieties of dimension zero}

By a \emph{zero-dimensional Shimura }$p$\emph{-datum}, we mean a pair $(T,X)$
in which $T$ is a torus over $\mathbb{Z}{}_{(p)}$ and $X$ is a finite set of
homomorphisms $\mathbb{S}{}\rightarrow T_{\mathbb{R}{}}$ on which
$T(\mathbb{R}{})/T(\mathbb{R}{})^{+}$ acts transitively. Consistent with our
standing assumptions for Shimura data, we require (except in \S 6) that the
weights of the elements of $X$ are defined over $\mathbb{Q}{}$ and that $T$
splits over $\mathbb{Q}{}^{\mathrm{cm}}$. Then, as in the preceding
subsection, each element $x$ of $X$ defines a $T$-motive $M_{x}$ satisfying
the condition (\ref{eq4}), p\pageref{eq4}, to which we can attach a set
$\mathcal{L}{}(M_{x})$ with an action of $T(\mathbb{A}{}_{f})$ and of $\Phi$.

The zero-dimensional Shimura variety attached to $(T,X)$ is as defined in
\cite{milne2005}, \S 5. Every Shimura $p$-datum $(G,X)$ defines a
zero-dimensional Shimura $p$-datum $(G^{\mathrm{ad}},X^{\mathrm{ad}})$, and
$\Sh_{p}(G^{\mathrm{ad}},X^{\mathrm{ad}})(\mathbb{C}{})=\pi_{0}(\Sh_{p}(G,X))$
when $G^{\mathrm{der}}$ is simply connected (ibid.).

\textit{From now on \textquotedblleft Shimura }$p$%
\textit{-datum\textquotedblright\ will mean either a Shimura }$p$%
\textit{-datum, as defined in the Introduction, or a zero-dimensional Shimura
}$p$\textit{-datum as just defined.}

\subsection{Statement of Conjecture LR+}

Define%
\begin{equation}
\mathcal{L}(G,X)=\bigsqcup_{m}\mathcal{L}(m) \label{e1}%
\end{equation}
where $m$ runs over the set of equivalence classes of special $G$-motives.
Then $\mathcal{L}(G,X)$ is a set with an action of $G(\mathbb{A}_{f}%
^{p})\times Z(\mathbb{Q}_{p})\times\{\Phi\}$, and
\[
(G,X)\rightsquigarrow\mathcal{L}(G,X)
\]
is a functor from the category of Shimura $p$-data to the category of sets
with a Frobenius operator.

Let $E_{v}$ be the closure of $E$ in $(\mathbb{Q}^{\mathrm{al}})_{v}$ and let
$\mathcal{O}_{v}$ be its ring of integers. Let $\Sh_{p}(G,X)$ be the canonical
integral model\label{cim} over $\mathcal{O}_{v}$ (in the sense of
\cite{milne1992})\footnote{Because of an error in \cite{faltingsC1990}, V 6.8,
the definition in \cite{milne1992} needs to be slightly modified --- see
\cite{milne1994s}, p513, and \cite{moonen1998}, \S 3.} of the Shimura variety
with complex points%
\[
G(\mathbb{Q})\backslash\left(  X\times G(\mathbb{A}_{f}^{p})\times
G(\mathbb{Q}_{p})/G(\mathbb{Z}_{p})\right)  .
\]
From its definition, $\Sh_{p}$ is uniquely determined, and it is known to
exist except possibly for $p=2$ (\cite{vasiu1999, vasiu2008, vasiu2008a},
\cite{kisin2007, kisin2009}). Write $\Sh_{p}(\mathbb{F})$ for the functor%
\[
(G,X)\rightsquigarrow\Sh_{p}(G,X)(\mathbb{F})
\]
from the category of Shimura $p$-data to the category of sets with a Frobenius operator.

\begin{lrconjecture}
\label{a15}There exists a canonical isomorphism of functors
\[
\mathcal{\mathcal{L}}\rightarrow\Sh_{p}(\mathbb{F})
\]
such that, for each Shimura $p$-datum $(G,X)$, the isomorphism%
\[
\mathcal{L}(G,X)\rightarrow\Sh_{p}(G,X)(\mathbb{F})
\]
is equivariant for the actions of $G(\mathbb{A}_{f}^{p})$ and $Z(\mathbb{Q}%
_{p}$).
\end{lrconjecture}

In symbols:%
\[
\renewcommand{\arraystretch}{1.5}%
\begin{array}
[c]{ccccc}%
(G,X) &  & \mathcal{L}(G,X) & \overset{1:1}{\leftrightarrow} & \Sh_{p}%
(G,X)(\mathbb{F})\\
\downarrow & \rightsquigarrow & \downarrow & \Box & \downarrow\\
(G^{\prime},X^{\prime}) &  & \mathcal{L}(G^{\prime},X^{\prime}) &
\overset{1:1}{\leftrightarrow} & \Sh_{p}(G^{\prime},X^{\prime})(\mathbb{F})
\end{array}
\]

In fact, one can show that, for Shimura varieties of abelian type, there
exists at most one isomorphism of functors $\mathcal{L}{}\rightarrow
\Sh_{p}(\mathbb{F}{})$ taking a specific value on Siegel varieties and
zero-dimensional Shimura varieties.

Since $\mathcal{L}(G,X)$ includes terms corresponding only to special
homomorphisms, Conjecture \ref{a15} forces the following conjecture.

\begin{spconjecture}
\label{a16}Up to isogeny, every point on $\Sh_{p}(G,X)$ with coordinates in
$\mathbb{F}$ lifts to a special point on $\Sh_{p}(G,X)$ with coordinates in a
finite extension of $B(\mathbb{F})$.
\end{spconjecture}

Recall that, from the definition of the canonical integral model and Hensel's
lemma,%
\[
\Sh_{p}(B(\mathbb{F}))\overset{1:1}{\leftrightarrow}\Sh_{p}(W(\mathbb{F}%
))\overset{\text{onto}}{\longrightarrow}\Sh_{p}(\mathbb{F}).
\]
The special-points conjecture is proved in \cite{zink1983}, 2.7, for simple
abelian varieties of PEL-type;\footnote{Let $L$ be a simple finite-dimensional
algebra over $\mathbb{Q}{}$ with a positive involution $x\mapsto x^{\ast}$,
and let $p$ be a prime number such that
\par
\begin{itemize}
\item $L\otimes\mathbb{Q}_{p}$ is a product of matrix algebras, \
\par
\item $p$ is unramified in the centre $Z$ of $L$, and
\par
\item $(\mathcal{O}_{L}\otimes\mathbb{Z}_{p})^{\ast}=(\mathcal{O}_{L}%
\otimes\mathbb{Z}_{p})$.
\end{itemize}
\par
\noindent Let $k$ be a field of characteristic $p$, and let $(A,\lambda)$ be a
polarized abelian $L$-variety over $k$ such that
\[
\Tr_{\mathbb{C}}(l\mid\Lie(A))=\sum_{\rho:Z\rightarrow\mathbb{C}}r_{\rho}%
\cdot\rho(\mathrm{Trd}(l)),\quad l\in L,\rlap{\hspace{1in}(*)}
\]
for some fixed integers $r_{\rho}$. Assume that the degree of $\lambda$ is
prime to $p$. Then Zink proves the following:\bquote Let $R$ be a product of
CM-fields and let $\theta_{A}\colon R\rightarrow\End_{L}^{0}(A)$ be a
homomorphism such that $\theta_{A}(R)$ is stable under the Rosati involution
and
\[
\dim_{\mathbb{Q}{}}(R)=2\dim A/[L\colon Z]^{1/2}%
\]
(that is, $A$ admits complex multiplication $(R,\theta_{A})$ relative to $L$).
Then there exists a discrete valuation ring $\mathcal{O}$ that is a finite
extension of $W(k)$ and a polarized abelian $\mathcal{O}_{L}$-variety
$(\tilde{A},\tilde{\lambda})$ with complex multiplication $(R,\theta
_{\tilde{A}})$ over $\mathcal{O}$ satisfying (*) whose reduction is isogenous
to $(A,\lambda,R,\theta_{A})$.\equote}\thinspace\ \footnote{According to Zink
(1983, p103), in this case the result was originally stated in a letter from
Langlands to Rapoport, but the proof there, which is based on the methods of
Grothendieck and Messing, is incorrect. (Dieses Resultat wurde in einem Brief
von Langlands an Rapoport behauptet. Der Beweis dort, der auf Methoden von
Grothendieck und Messing beruht ist aber fehlerhaft.)} more general results
have been announced by \citet{vasiu2003cm}. It should be noted that the
special-points conjecture is false in the case of bad reduction
(\cite{langlandsR1987}).

\subsection{Example: the Shimura variety of $\mathbb{G}_{m}$}

In order to check our signs, we prove the conjecture for a Shimura $p$-datum
$(G,X)$ with $G\ $a one-dimensional split torus over $\mathbb{Z}{}_{(p)}$. Any
such Shimura $p$-datum is of the form $(\mathbb{G}_{m},\{h_{n}\})$ where
$h_{n}$ is the $n$th power of the norm map $\mathbb{S}\rightarrow
\mathbb{G}_{m\mathbb{R}}$; thus $h_{n}(z)=(z\bar{z})^{n}$ for $z\in\mathbb{C}%
$. The cocharacter $\mu_{n}$ of $\mathbb{G}_{m}$ attached to $h_{n}$ is
$z\mapsto z^{n}$.

The maps%
\[
\Sh_{p}(\mathbb{C})\leftarrow\Sh_{p}(\mathbb{Q}^{\mathrm{al}})\rightarrow
\Sh_{p}((\mathbb{Q}^{\mathrm{al}})_{v})\leftarrow\Sh_{p}(B(\mathbb{F}%
))\leftarrow\Sh_{p}(W(\mathbb{F}))\rightarrow\Sh_{p}(\mathbb{F})
\]
are bijective because $\Sh_{p}$ is of dimension zero and pro-\'{e}tale over
$\mathbb{Z}_{p}$. Therefore%
\[
\Sh_{p}(\mathbb{F})\simeq\mathbb{Q}^{\times}\backslash\mathbb{A}_{f}^{p}%
\times\mathbb{Q}_{p}^{\times}/\mathbb{Z}_{p}^{\times}.
\]
The Frobenius automorphism $x\mapsto x^{p}$ of $\mathbb{F}$ acts on this as
multiplication by $p^{-n}$ on $\mathbb{Q}_{p}^{\times}$ (see \cite{milne1992}, \S 1).

Let $\xi$ denote a one-dimensional representation of $\mathbb{G}_{m}$ with
character $x\mapsto x$. Up to equivalence, the only special $\mathbb{G}_{m}%
$-motive $M$ over $\mathbb{F}$ is that with $M(\xi)=\1(n)$, the $n$th tensor
power of the Tate motive. We have
\begin{align*}
I(M)  &  =\Aut(\1(n))\simeq\mathbb{Q}^{\times},\\
X^{p}(M)  &  =\Hom(\mathbb{A}_{f}^{p}\otimes V(\xi),\mathbb{A}_{f}^{p}(n)),
\end{align*}
and $X_{p}(M)$ is equal to the set of lattices $\Lambda$ in $\omega
_{p}(\1(n))$ that are strongly divisible for the filtration defined by
$\mu_{n}^{-1}$. The Frobenius operator $\Phi$ acts by sending $\Lambda$ to
$F\Lambda$.

As an $\mathbb{A}_{f}^{p}$-module, $\mathbb{A}_{f}^{p}(n)=\mathbb{A}_{f}^{p}$,
and so the choice of a basis element for $V(\xi)$ determines a bijection
$X^{p}(M)\rightarrow\mathbb{A}_{f}^{p}$.

The isocrystal $\omega_{p}(\1(n))$ is $B(\mathbb{F})$ with $F$ acting as
$p^{-n}\sigma$. Every strongly divisible lattice $\Lambda$ in $\omega
_{p}(\1(n))$ arises from a lattice $\Lambda_{0}$ in $\mathbb{Q}_{p}$
(cf.\ \cite{wintenberger1984}, 4.2.5(1)). Specifically, to say that $\Lambda$
is strongly divisible means that $\mu_{n}(p^{-1})\Lambda=F\Lambda$, i.e.,
$p^{-n}\Lambda=p^{-n}\sigma\Lambda$; therefore $\Lambda=\sigma\Lambda$, and so
$\Lambda^{\sigma=1}$ is a lattice in $\mathbb{Q}_{p}$. Let $\Lambda^{\sigma
=1}=a(\Lambda)\cdot\mathbb{Z}_{p}$ with $a(\Lambda)\in\mathbb{Q}_{p}$; then
$\Lambda\mapsto a(\Lambda)$ defines a bijection $X_{p}(M)\rightarrow
\mathbb{Q}_{p}/\mathbb{Z}_{p}$ under which the Frobenius maps correspond.

This completes the proof of Conjecture \ref{a15} for $(\mathbb{G}_{m}%
,\{h_{n}\})$.

\subsection{A criterion for a $G$-motive to be special}

We begin by reviewing some constructions.

\begin{plain}
\label{a17}A point $x$ of $X$ defines an exact tensor functor
\[
\xi\rightsquigarrow(V(\xi)_{\mathbb{R}},\xi_{\mathbb{R}}\circ h_{x}%
)\colon\Rep_{\mathbb{Q}}(G)\rightarrow\Hdg_{\mathbb{R}},
\]
and hence a $G$-object $N_{x}$ of $\mathsf{R}_{\infty}$. If $x^{\prime}=gx$
with $g\in G(\mathbb{R})$, then $g$ defines an isomorphism $N_{x}\rightarrow
N_{x^{\prime}}$, and so the equivalence class of $N_{x}$ depends only on $X$.
\end{plain}

\begin{plain}
\label{a18}For any torus $T$ split by $\mathbb{Q}^{\mathrm{cm}}$ and
cocharacter $\mu$ with weight $-(\iota+1)\mu$ defined over $\mathbb{Q}$, there
is a unique homomorphism $S\rightarrow T$ sending $\mu^{S}$ to $\mu$
(universal property of $(S,\mu^{S})$). This homomorphism defines a functor
$\Rep(T)\rightarrow\Rep(S)\simeq\CM(\mathbb{Q}^{\text{al}})$,i.e., a
$T$-object $N(T,\mu)$ of $\CM(\mathbb{Q}^{\text{al}})$. For example, from
$(G,X)$, we get a $G^{\mathrm{ab}}$-object $N(G^{\mathrm{ab}},\mu_{X})$ in
$\CM(\mathbb{Q}^{\text{\textrm{al}}})$ where $\mu_{X}$ is the common composite
of $\mu_{x}$ with $G\rightarrow G^{\mathrm{ab}}$ for $x\in X$.
\end{plain}

\begin{proposition}
\label{a19}When $G^{\mathrm{der}}$ is simply connected, a $G$-motive
$M\colon\Rep_{\mathbb{Q}}(G)\rightarrow\Mot(\mathbb{F})$ is special if and
only if it satisfies the following conditions:

\begin{enumerate}
\item the $G_{\mathbb{R}}$-object $M_{(\infty)}$ is equivalent to $N_{x}$
($x\in X$);

\item for all $\ell\neq p,\infty$, the $G_{\mathbb{Q}_{\ell}}$-object
$M_{(\ell)}$ is equivalent to $\mathbb{Q}_{\ell}\otimes_{\mathbb{Q}}V$ (here
$V$ is the forgetful functor $\xi\mapsto V(\xi)$);

\item the tensor functors $M_{(p)}$ and $B(\mathbb{F})\otimes_{\mathbb{Q}}V$
are equivalent, and there exists a $p$-integral structure on $M$;

\item the $G^{\mathrm{ab}}$ object in $\Mot(\mathbb{F)}$ obtained from $M$ by
restriction is equivalent to $N(G^{\mathrm{ab}},\mu_{X})$.
\end{enumerate}
\end{proposition}

\begin{proof}
The proof is similar to that of \cite{langlandsR1987}, 5.3, p173.
\end{proof}

\subsection{Re-statement of the conjecture in terms of motives with
$\mathbb{Z}{}_{(p)}$-coefficients}

By a $G$\emph{-motive over} $\mathbb{F}$ \emph{with coefficients in}
$\mathbb{Z}_{(p)}$, we mean an exact tensor functor
\[
M\colon\Rep(G)\rightarrow\Mot_{(p)}(\mathbb{F}),\quad\xi\rightsquigarrow
(M_{p}(\xi),M_{0}(\xi),m(\xi)).
\]
We say that $M$ is \emph{admissible} if the image of $m(\xi)$ is a
$p$-structure on the $G$-motive $\xi\rightsquigarrow M_{0}(\xi)$. For an
admissible $M$, let%
\begin{align*}
I(M)  &  =\Aut^{\otimes}(M)\quad\quad\text{(}=\Aut^{\otimes}(M_{0}%
)\cap\Aut^{\otimes}(M_{p})\text{),}\\
X^{p}(M)  &  =X^{p}(M_{0}).
\end{align*}
The group $I(M)$ acts on $X^{p}(M)$, and so we can define%
\[
\mathcal{L}(M)=I(M)\backslash X^{p}(M)/Z^{p}.
\]
An admissible $G$-motive $M$ over $\mathbb{F}$ with coefficients in
$\mathbb{Z}_{(p)}$ is special if $M_{0}$ is special. For an equivalence class
$m$ of special admissible $G$-motives with coefficients in $\mathbb{Z}_{(p)}$,
define%
\[
\mathcal{L}(m)=\varprojlim_{M\in m}\mathcal{L}(M).
\]

\begin{proposition}
\label{a20}We have%
\begin{equation}
\mathcal{L}(G,X)\simeq\bigsqcup\nolimits_{m}\mathcal{L}(m) \label{e2}%
\end{equation}
where $m$ runs over the set of equivalence classes of special admissible
$G$-motives with coefficients in $\mathbb{Z}_{(p)}$.
\end{proposition}

More precisely, let $M=(M_{p},M_{0},m)$ be a special $\mathbb{Z}{}_{(p)}%
$-motive over $\mathbb{F}$ with $G$-structure. There is an obvious map
$\mathcal{L}(M)\rightarrow\mathcal{L(}M_{0})$, and these maps induce an
isomorphism (\ref{e2}). The proof of this is straightforward using that
$G(\mathbb{Q}_{p})=G(\mathbb{Q})\cdot G(\mathbb{Z}_{p})$ (\cite{milne1994s}, 4.9).

\section{The functor of points defined by a Shimura variety}

Throughout this section $(G,X)$ is a Shimura $p$-datum such that
$(G_{\mathbb{Q}},X)$ is of abelian type in the sense of \cite{milneS1982},
\S 1.\footnote{Recall that a Shimura datum $(G,X)$ is said to be of
\emph{abelian type} if there exists an isogeny $H\rightarrow G^{\mathrm{der}}$
with $H$ a product of almost-simple groups $H_{i}$ such that either
\par
\begin{enumerate}
\item $H_{i}$ is simply connected of type $A$, $B$, $C$, or $D^{\mathbb{R}}$,
or,
\par
\item $H_{i}$ is of type $D_{n}^{\mathbb{H}}$ $\;($$n\geq5$) and equals
$\Res_{F/\mathbb{Q}}H^{\prime}$ for $F$ a totally real field and $H^{\prime}$
is a form of $\SO(2n)$ (double covering of the adjoint group).
\end{enumerate}
\par
\noindent} Moreover, we assume that $\ad h_{x}(i)$ is a Cartan involution on
$G_{\mathbb{R}}/w_{X}(\mathbb{G}_{m})$ for one (hence all) $x\in X$, and we
fix a homomorphism $t\colon G_{\mathbb{Q}}\rightarrow\mathbb{G}_{m}$ (assumed
to exist) such that $t\circ w_{X}=-2$. Recall that the condition on $\ad
h_{x}(i)$ implies that $Z(\mathbb{Q})$ is discrete in $Z(\mathbb{A}_{f})$ and
that $Z(\mathbb{Z}_{(p)})$ is discrete in $Z(\mathbb{A}_{f}^{p})$ (e.g.,
\cite{milne2005}, 5.26).

As before, $\Sh_{p}$ denotes the canonical integral model of $\Sh(G,X)$ over
$\mathcal{O}_{v}$ (see p\pageref{cim}).

For a field $k$ of characteristic zero, $\Mot(k)$ denotes the category of
motives over $k$ based on abelian varieties and using the Hodge classes as
correspondences. It is a semisimple tannakian category over $\mathbb{Q}$ whose
objects are the abelian motives over $k$. We shall simply call them motives. A
$G$-motive over $k$ is an exact tensor functor $\Rep_{\mathbb{Q}%
}(G)\rightarrow\Mot(k)$.

When $k=\mathbb{C}$, Betti cohomology defines a $\mathbb{Q}$-valued fibre
functor $\omega_{B}$. Etale cohomology defines fibre functors $\omega
_{l}\colon\Mot(k)\rightarrow\Vc_{\mathbb{Q}_{l}}$ for all primes $l\neq\infty
$, and an exact tensor functor $\omega_{f}^{p}\colon\Mot(k)\rightarrow
\Mod_{\mathbb{A}_{f}^{p}}$. Strictly speaking, these depend on the choice of
an algebraic closure $k^{\mathrm{al}}$ of $k$, and $\omega_{l}(M)$ is a
$\Gal(k^{\mathrm{al}}/k)$-module for each $M$.

\subsection{Definition of the functor $\mathcal{M}$}

\subsubsection{Admissible $G$-motives.}

Let $E$ be the reflex field $E(G,X)$ of $(G,X)$, and let $k$ be a field
containing $E$. As before, a point $x$ of $X$ defines an exact tensor functor
\[
H_{x}\colon\Rep_{\mathbb{Q}{}}(G)\rightarrow\Hdg_{\mathbb{Q}},\quad
\xi\rightsquigarrow(V(\xi),\xi_{\mathbb{R}}\circ h_{x}).
\]
An $E$-homomorphism $\tau\colon k\rightarrow\mathbb{C}$ defines an exact
tensor functor
\[
\omega_{\tau}\colon\Mot(k)\rightarrow\Hdg_{\mathbb{Q}},\quad X\mapsto
\omega_{B}(\tau X)\text{.}%
\]
We say that a $G$-motive $M$ over $k$ is \emph{admissible with respect to
}$\tau$ if $\omega_{\tau}\circ M$ is isomorphic to $H_{x}$ for some $x\in X.$

\begin{proposition}
\label{a21}If $M$ is admissible with respect to one $E$-homomorphism
$k\rightarrow\mathbb{C}$, then it is admissible with respect to every
$E$-homomorphism $k\rightarrow\mathbb{C}$.
\end{proposition}

\begin{proof}
The proof is the same as that of \cite{milne1994s}, 3.29.
\end{proof}

\noindent\begin{minipage}[b]{4.0in}
We say that a $G$-motive $M$ over a field $k$ containing $E$ is
\emph{admissible} if there exists a $G$-motive $M_{0}$ over a subfield $k_{0}$
of $k$ containing $E$ such that
\begin{itemize}
\item $M_{0}$ gives rise to $M$ by extension of
scalars and
\item $M_{0}$ is admissible with respect to some homomorphism
$k_{0}\rightarrow\mathbb{C}$.\end{itemize}\end{minipage}\hfill
\raisebox{0.5in}{$\bfig
\node a(0,0)[E]
\node b(0,300)[k_0]
\node c(-400,350)[k]
\node d(0,600)[M_0]
\node e(-400,650)[M]
\node g(400,650)[H_x]
\node f(400,350)[\mathbb{C}]
\arrow[a`b;{}]
\arrow[b`c;{}]
\arrow/-/[b`d;{}]
\arrow/-/[c`e;{}]
\arrow/-/[d`e;{}]
\arrow[a`f;{}]
\arrow[b`f;{}]
\arrow/-/[d`g;{}]
\arrow/-/[f`g;{}]
\efig$}

\subsubsection{Etale $p$-integral structures.}

Let $k$ be a field containing $E$, and let $\Gamma=\Gal(k^{\mathrm{al}}/k)$
for some algebraic closure $k^{\mathrm{al}}$ of $k$. For a $G$-motive $M$ over
$k$, let $M(p)\colon\Rep_{\mathbb{Q}_{p}}(G)\rightarrow\Mot(k)_{(\mathbb{Q}%
_{p})}$ be the exact tensor functor obtained by extension of scalars
$\mathbb{Q}\rightarrow\mathbb{Q}_{p}$.

\begin{definition}
\label{a22}An \emph{\'{e}tale} $p$\emph{-integral structure\/} on a $G$-motive
$M\colon\Rep_{\mathbb{Q}}(G)\rightarrow\Mot(k)$ is an exact tensor functor
\underline{$\Lambda$}$_{p}\colon\Rep_{\mathbb{Z}_{p}}(G_{p})\rightarrow
\Rep_{\mathbb{Z}_{p}}(\Gamma)$ such that, for all $\xi$ in $\Rep_{\mathbb{Z}%
_{p}}(G)$, $\underline{\Lambda}_{p}(\xi)$ is a $\mathbb{Z}_{p}$-lattice in
$(\omega_{p}\circ M(p))(\xi_{\mathbb{Q}_{p}})$.
\end{definition}

The condition means that $\underline{\Lambda}_{p}(\xi)\subset(\omega_{p}\circ
M(p))(\xi_{\mathbb{Q}_{p}})$ for all $\xi$, and there is a commutative
diagram
\[
\xymatrix{
\Rep_{\mathbb{Q}_{p}}(G)\ar[r]^{M(p)}
&\Mot(\mathbb{F})_{\left(\mathbb{Q}_{p}\right)}\ar[r]^{\omega_{p}}
&\Rep_{\mathbb{Q}_p}(\Gamma)\\
\Rep_{\mathbb{Z}_{p}}(G)\ar[u]_{\mathbb{Q}_p\otimes_{\mathbb{Z}_p}-}\ar[rr]^{\underline{\Lambda}_p}
&&\Rep_{\mathbb{Z}_p}(\Gamma)\ar[u]_{\mathbb{Q}_p\otimes_{\mathbb{Z}_p}-}.
}
\]

\begin{lemma}
\label{a23}For every exact tensor functor $\underline{\Lambda}_{p}%
\colon\Rep_{\mathbb{Z}_{p}}(G)\rightarrow\Rep_{\mathbb{Z}_{p}}(\Gamma)$ there
exists an isomorphism $\Lambda\rightarrow\omega_{\mathrm{forget}}%
\circ\underline{\Lambda}_{p}$ of tensor functors. Here $\Lambda$ denotes the
forgetful functor $\xi\mapsto\Lambda(\xi)\colon\Rep_{\mathbb{Z}_{p}%
}(G)\rightarrow\Mod_{\mathbb{Z}_{p}}.$
\end{lemma}

\begin{proof}
Apply (\ref{a1}).
\end{proof}

\subsubsection{The functor $\mathcal{M}$.}

Let $k$ be a field containing $E$. For an admissible $G$-motive $M$ over $k$,
define%
\begin{align*}
I(M)  &  =\Aut^{\otimes}(M),\\
X^{p}(M)  &  =\Isom^{\otimes}(\mathbb{A}_{f}^{p}\otimes V,\omega_{f}^{p}\circ
M),\\
X_{p}(M)  &  =\{\text{\'{e}tale }p\text{-integral structures on }M\}.
\end{align*}
Then $I(M)$ acts on $X^{p}(M)$ and $X_{p}(M)$ on the left, $G(\mathbb{A}%
_{f}^{p})$ acts on $X^{p}(M)$ on the right, and $Z_{p}(\mathbb{Q}{}_{p})$ acts
on $X_{p}(M)$. We define%
\[
\mathcal{M}(M)=I(M)\backslash X^{p}(M)\times X_{p}(M)
\]
regarded as a $G(\mathbb{A}_{f}^{p})\times Z(\mathbb{Q}_{p})$-set. An
isomorphism $M\rightarrow M^{\prime}$ defines an equivariant bijection
$\mathcal{M}(M)\rightarrow\mathcal{M}(M^{\prime})$ which is independent of the
isomorphism. For an equivalence class $m$ of $G$-motives over $k$, we define%
\[
\mathcal{M}(m)=\varprojlim_{M\in m}\mathcal{M}(M).
\]
Define%
\[
\mathcal{M}(G,X)(k)=\bigsqcup\nolimits_{m}\mathcal{M}(m)
\]
where $m$ runs over the equivalence classes of admissible $G$-motives. This is
a set with an action of $G(\mathbb{A}_{f}^{p})\times Z(\mathbb{Q}_{p})$. For a
fixed $k$, it is a functor from the category of Shimura $p$-data to the
category of sets, and, for a fixed $(G,X)$, it is a functor from the category
of $E$-fields to the category of sets endowed with an action of $G(\mathbb{A}%
_{f}^{p})\times Z(\mathbb{Q}_{p}$).

\subsection{The points of $\mathrm{Sh}_{p}$ with coordinates in the complex
numbers}

With our assumptions%
\begin{align*}
\Sh_{p}(\mathbb{C})  &  \simeq G(\mathbb{Z}_{(p)})\backslash X\times
G(\mathbb{A}_{f}^{p})\\
&  \simeq G(\mathbb{Q})\backslash X\times G(\mathbb{A}_{f}^{p})\times
(G(\mathbb{Q}_{p})/G(\mathbb{Z}_{p}))
\end{align*}
(\cite{milne1994s}, 4.11, 4.12).

Let $M$ be an admissible $G$-motive over $\mathbb{C}$, and let $(\eta
,\underline{\Lambda}_{p})\in X^{p}(M)\times X_{p}(M)$. Because $M$ is
admissible, there exists an isomorphism of tensor functors $\beta\colon
\omega_{B}\circ M\rightarrow V$ sending $h_{M}$ to $h_{x}$ for some $x\in X$.
When we tensor this with $\mathbb{A}{}_{f}^{p}$ and compose with $\eta$,%
\[
\xymatrixcolsep{3.5pc}\xymatrix{
\mathbb{A}_{f}^{p}\otimes_{\mathbb{Q}}V\ar[r]_-{\eta}\ar@/^1.5pc/[rrr]^{g(\eta)}&\omega_{f}^{p}\circ M\ar[r]^-{\simeq}
&\mathbb{A}_{f}^{p}\otimes_{\mathbb{Q}}(\omega_{B}\circ M)
\ar[r]_-{\mathbb{A}_{f}^{p}\otimes_{\mathbb{Q}}\beta}
&\mathbb{A}_{f}^{p}\otimes_{\mathbb{Q}}V,
},
\]
we get an element of $\underline{\Aut}^{\otimes}(\mathbb{A}_{f}^{p}%
\otimes_{\mathbb{Q}}V)\simeq G(\mathbb{A}_{f}^{p})$. The $p$-integral
structure $\underline{\Lambda}_{p}$ is transformed by $\beta$ into a
$p$-integral structure on the forgetful functor $V\colon\Rep_{\mathbb{Q}_{p}%
}(G)\rightarrow\Vc_{\mathbb{Q}_{p}}$, and it follows from (\ref{a1}) there
exists a $g=g(\underline{\Lambda}_{p})\in G(\mathbb{Q}_{p})$ such that the
isomorphism%
\[
\beta(\xi)\colon(\omega_{p}\circ M(p))(\xi_{\mathbb{Q}_{p}})\longrightarrow
V(\xi_{\mathbb{Q}_{p}})
\]
maps $\underline{\Lambda}_{p}(\xi)$ onto $g\cdot\Lambda(\xi)$ for all $\xi$ in
$\Rep_{\mathbb{Z}_{p}}(G)$. Since $\beta$ is uniquely determined up to an
element of $G(\mathbb{Q})$, we get a well-defined map%
\begin{equation}
\lbrack\eta,\underline{\Lambda}_{p}]\mapsto\lbrack x,\beta\circ\eta
,g(\underline{\Lambda}_{p})]\colon\mathcal{M}(M)\rightarrow\Sh_{p}%
(\mathbb{C}). \label{e4}%
\end{equation}
One checks the following statement as in\ \cite{milne1994s}, 4.14: for each
equivalence class $m$ of admissible $G$-motives, the maps (\ref{e4}) define an
injective map $\mathcal{M}(m)\overset{\textup{{\tiny def}}}{=}\varprojlim
_{m\in M}\mathcal{M}(M)\rightarrow\Sh_{p}(\mathbb{C})$, and the images of the
maps $\mathcal{M}(m)\rightarrow\Sh_{p}(\mathbb{C})$ for distinct $m$ are
disjoint and cover $\Sh_{p}(\mathbb{C}).$ In other words, the following is true.

\begin{proposition}
\label{a24}The above maps define a $G(\mathbb{A}_{f}^{p})\times Z(\mathbb{Q}%
_{p})$-equivariant bijection%
\[
\alpha(\mathbb{C})\colon\mathcal{M}(G,X)(\mathbb{C})\rightarrow\Sh_{p}%
(G,X)(\mathbb{C}).
\]

\end{proposition}

\subsection{The points of $\mathrm{Sh}_{p}$ with coordinates in a field of
characteristic zero}

The next result is a restatement of \cite{milne1994s}, 3.13.

\begin{theorem}
\label{a25}Let $k$ be a field containing $E$. For any $E$-homomorphism
$\tau\colon k\rightarrow\mathbb{C}$, the restriction of $\alpha(\mathbb{C})$
to $\mathcal{M}(k)$ factors through $\Sh_{p}(k)$:
\[
\begin{CD}
\mathcal{\mathcal{M}}(\mathbb{C}) @>{\alpha(\mathbb{C})}>> \Sh_{p}(\mathbb{C})\\
@AA{\tau}A@A{\bigcup}A{\tau}A\\
\mathcal{M}(k) @>{\alpha(k)}>> \Sh_{p}(k)
\end{CD}
\]
The map $\alpha(k)$ is a bijection which is independent of $\tau$, and it is
equivariant for the actions of $G(\mathbb{A}_{f}^{p})\times Z(\mathbb{Q}_{p})$
and $\Gal(k^{\mathrm{al}}/k)$.
\end{theorem}

For a fixed $k$, $\alpha(k)$ is an isomorphism of functors from the category
of Shimura $p$-data to sets, and, for a fixed $(G,X)$, it is an isomorphism of
functors from the category of $E$-fields to the category of sets with an
action of $G(\mathbb{A}_{f}^{p})\times Z(\mathbb{Q}_{p})$.

\subsection{The category of motives over a field of characteristic zero with
$\mathbb{Z}{}_{(p)}$-cofficients}

It is, perhaps, more natural to state Theorem \ref{a25} in terms of motives
with $\mathbb{Z}{}_{(p)}$-coefficients. Let $k$ be a field of characteristic zero.

A \emph{motive} $M$ \emph{over} $k$ \emph{with coefficients in }$\mathbb{Z}%
{}_{(p)}$ is a triple $(M_{p},M_{0},m)$ consisting of

\begin{enumerate}
\item a finitely generated $\mathbb{Z}_{p}$-module $M_{p}$ equipped with a
continuous action of $\Gal(k^{\mathrm{al}}/k)$,

\item an object $M_{0}$ of $\Mot(\mathbb{F})$, and

\item an isomorphism $m\colon(M_{p})_{\mathbb{Q}}\rightarrow\omega
_{p}^{\mathbb{F}}(M_{0})$.
\end{enumerate}

\noindent A morphism $\alpha\colon M\rightarrow N$ of motives with
$\mathbb{Z}{}_{(p)}$-coefficients is a pair of morphisms%
\[
(\alpha_{p}\colon M_{p}\rightarrow N_{p},\alpha_{0}\colon M_{0}\rightarrow
N_{0})
\]
such that $n\circ\alpha_{p}=\omega_{p}^{\mathbb{F}}(\alpha_{0})\circ m$. Let
$\Mot_{(p)}(k)$ be the category of motives over $k$ with coefficients in
$\mathbb{Z}{}_{(p)}$, and let $\Mot_{(p)}^{\prime}(k)$ be the full subcategory
of objects $M$ such that $M_{p}$ is torsion-free. The objects of
$\Mot_{(p)}^{\prime}(k)$ will be called torsion-free.

\begin{proposition}
\label{a26}With the obvious tensor structure, $\Mot_{(p)}(k)$ is an abelian
tensor category over $\mathbb{Z}_{(p)}$, and $\Mot_{(p)}^{\prime}(k)$ is a
rigid pseudo-abelian tensor subcategory of $\Mot_{(p)}(k)$. Moreover:

\begin{enumerate}
\item the tensor functor $M\rightsquigarrow M_{p}\colon\Mot_{(p)}%
(\mathbb{F})\rightarrow\Mod_{\mathbb{Z}_{p}}$ is exact;

\item there are canonical equivalences of categories
\[
\Mot_{(p)}^{\prime}(k)_{\mathbb{Q}}\rightarrow\Mot_{(p)}(k)_{\mathbb{Q}%
}\rightarrow\Mot(k);
\]

\item for all torsion-free motives $M,N$, the cokernel of%
\[
\Hom(M,N)\otimes_{\mathbb{Z}_{(p)}}\mathbb{Z}_{p}\rightarrow\Hom(M_{p},N_{p})
\]
is torsion-free.
\end{enumerate}
\end{proposition}

\begin{proof}
Routine.
\end{proof}

By a $G$\emph{-motive over }$k$ \emph{with coefficients in }$\mathbb{Z}%
{}_{(p)}$, we mean an exact tensor functor%
\[
M\colon\Rep(G)\rightarrow\Mot_{(p)}(\mathbb{F}{}),\quad\quad\xi
\rightsquigarrow(M_{p}(\xi),M_{0}(\xi),m(\xi)).
\]
We say that $M$ is \emph{admissible} if $\xi\rightsquigarrow M_{0}(\xi)$ is
admissible and the image of $m(\xi)$ is an \'{e}tale $p$-integral structure on
$M_{0}$. For an admissible $M$, let%
\begin{align*}
I(M)  &  =\Aut^{\otimes}(M),\text{ and}\\
X^{p}(M)  &  =X^{p}(M_{0}).
\end{align*}
The group $I(M)$ acts on $X^{p}(M)$, and so we can define%
\[
\mathcal{L}(M)=I(M)\backslash X^{p}(M).
\]
For an equivalence class $m$ of admissible $G$-motives over $k$ with
coefficients in $\mathbb{Z}{}_{(p)}$, define%
\[
\mathcal{L}(m)=\varprojlim_{M\in m}\mathcal{L}(M).
\]
Then%
\[
\mathcal{L}(G,X)\simeq\bigsqcup\nolimits_{m}\mathcal{L}(m)
\]
where $m$ runs over the equivalence classes of admissible $G$-motives with
coefficients in $\mathbb{Z}{}_{(p)}$.

\section{The map $\mathcal{M}(G,X)(W(\mathbb{F}))\rightarrow\mathcal{L}(G,X)$}

Let $(G,X)$ be a Shimura $p$-datum, and $\Sh_{p}(G,X)$ be a canonical integral
model. We often write $?$ for $?(G,X)$. Conjecturally, there should be a map
\[
\mathcal{M}(B(\mathbb{F}))\rightarrow\mathcal{L}(\mathbb{F})
\]
corresponding to the map%
\[
\Sh_{p}(B(\mathbb{F}))\simeq\Sh_{p}(W(\mathbb{F}))\twoheadrightarrow
\Sh_{p}(\mathbb{F}),
\]
but (at present) we are able to define such a map only on a subset of
$\mathcal{M}(B(\mathbb{F}))$.\footnote{The problem is that we defined
$\Mot(\mathbb{F}{})$ as a quotient of $\CM(\mathbb{Q}^{\mathrm{al}})$, but we
would like to realize it as a quotient of the category of abelian motives over
$\mathbb{Q}{}^{\mathrm{al}}$ having good reduction at $v$, or, at least, of
the category generated by abelian varieties over $\mathbb{Q}{}^{\mathrm{al}}$
having good reduction at $v$. The latter would be possible if we knew the
rationality conjecture for all abelian varieties with good reduction at $v$.}

\subsection{The points of $\mathrm{Sh}_{p}$ with coordinates in $B(\mathbb{F}%
)$}

Theorem \ref{a25} gives, in particular, a motivic description of the points of
$\Sh_{p}$ with coordinates in $B(\mathbb{F})$. However, as we explained in
\cite{milne1994s}, p509, because \'{e}tale $p$-integral structures do not
reduce well, to pass from the points on a Shimura variety with coordinates in
$B(\mathbb{F})$ to the points with coordinates in $\mathbb{F}$, we need to
replace the \'{e}tale $p$-integral structures with crystalline $p$-integral structures.

\begin{lemma}
\label{a27} For any admissible $G$-motive $M$ over $B(\mathbb{F})$, there
exists an isomorphism
\[
B(\mathbb{F})\otimes_{\mathbb{Q}}V\rightarrow\omega_{\mathrm{dR}}\circ
M\quad\quad
\]
of tensor functors $\Rep_{\mathbb{Q}}(G)\rightarrow\Vc_{B(\mathbb{F})}$
carrying the filtration defined by $\mu_{0}^{-1}$ into the de Rham filtration
(here $\mu_{0}$ is as in \ref{a11} and $V$ is the forgetful functor).
\end{lemma}

\begin{proof}
For each $B(\mathbb{F})$-algebra $R$, let $\mathcal{F}(R)$ be the set of
isomorphisms of $R$-linear tensor functors
\[
R\otimes_{\mathbb{Q}}V\rightarrow R\otimes_{B(\mathbb{F})}(\omega
_{\mathrm{dR}}\circ M)
\]
carrying $\Filt(\mu_{0}^{-1})$ into the de Rham filtration. Then $\mathcal{F}$
is a pseudo-torsor for the subgroup $P$ of $G_{B(\mathbb{F})}$ respecting the
filtration defined by $\mu_{0}^{-1}$ on each representation of $G$. This is a
parabolic subgroup of $G$ (\cite{saavedra1972}, IV 2.2.5, p223), and hence is
connected by a theorem of Chevalley (\cite{borel1991}, 11.16, p154). Once we
show $\mathcal{F}(\mathbb{C})\neq\emptyset$, so that $\mathcal{F}$ is a
torsor, it will follow from \cite{steinberg1965}, 1.9, that $\mathcal{F}%
(B(\mathbb{F}))\neq\emptyset$.

Choose an $E$-homomorphism $\tau\colon B(\mathbb{F})\rightarrow\mathbb{C}$.
There is a canonical comparison isomorphism
\[
\mathbb{C}\otimes_{\mathbb{Q}}(\omega_{B}\circ\tau M)\rightarrow
\mathbb{C}\otimes_{B(\mathbb{F})}(\omega_{\mathrm{dR}}\circ M)
\]
which carries the Hodge filtration on the left to the de Rham filtration on
the right. By assumption, for some $x\in X$, there exists an isomorphism
\[
H_{x}\rightarrow\omega_{B}\circ\tau M
\]
preserving Hodge structures. On combining these isomorphisms, we obtain an
isomorphism
\[
\mathbb{C}\otimes_{\mathbb{Q}}H_{x}\rightarrow\mathbb{C}\otimes_{B(\mathbb{F}%
)}(\omega_{\mathrm{dR}}\circ M)
\]
carrying $\Filt(\mu_{x}^{-1})$ to the Hodge filtration. By its very
definition, $\mu_{0}$ lies in the same $G(\mathbb{C})$-conjugacy class as
$\mu_{x}$, and so there exists an isomorphism of tensor functors
\[
\mathbb{C}\otimes V\rightarrow\mathbb{C}\otimes H_{x}%
\]
carrying $\Filt(\mu_{0}^{-1})$ to $\Filt(\mu_{x}^{-1})$. The composite of the
last two isomorphisms is an element of $\mathcal{F}(\mathbb{C})$.
\end{proof}

Let $\MF_{B(\mathbb{F})}$ denote the category of weakly admissible filtered
$B(\mathbb{F})$-modules (see \ref{a13}),\footnote{Thus an object of
$\MF_{B(\mathbb{F})}$ is an $F$-isocrystal over $\mathbb{F}$ together with a
finite exhaustive separated decreasing filtration that admits a strongly
divisible $W$-lattice.} and let $\MF_{W(\mathbb{F})}$ denote the category
$\underline{\mathrm{MF}}_{\mathrm{tf}}$ of \cite{fontaine1983}, 2.1. Thus an
object of $\MF_{W(\mathbb{F})}$ is a finitely generated $W(\mathbb{F})$-module
$\Lambda$ together with

\begin{enumerate}
\item a finite exhaustive separated decreasing filtration
\[
\cdots\supset\Filt^{i}\Lambda\supset\Filt^{i+1}\Lambda\supset\cdots
\]
by submodules that are direct summands of $\Lambda$, and

\item a family of $\sigma$-linear maps $\varphi_{\Lambda}^{i}\colon
\Filt^{i}\Lambda\rightarrow\Lambda$ such that $\varphi_{\Lambda}%
^{i}(x)=p\varphi_{\Lambda}^{i+1}(x)$ for $x\in\Filt^{i+1}\Lambda$ and
$\sum_{i}\im\varphi_{\Lambda}^{i}=\Lambda$.
\end{enumerate}

\noindent Let $\Lambda$ be an object of $\MF_{W(\mathbb{F})}$. Then
$\Lambda_{B(\mathbb{F})}$ is an object of $\MF_{B(\mathbb{F})}$, and the image
of $\Lambda$ in $\Lambda_{B(\mathbb{F})}$ is a strongly divisible lattice such
that $\Filt^{i}\Lambda\overset{\textup{{\tiny def}}}{=}\Lambda\cap\Filt^{i}N$
is a direct summand of $\Lambda$ for all $i$. To give an object of
$\MF_{W(\mathbb{F})}$ that is torsion-free as a $W$-module is the same as
giving an object $(N,F,\Filt^{i})$ of $\MF_{B(\mathbb{F})}$ together with a
strongly divisible lattice $\Lambda$ such that $\Lambda\cap\Filt^{i}N$ is a
direct summand of $\Lambda$ for all $i$. The category $\MF_{W(\mathbb{F})}$ is
a $\mathbb{Z}_{p}$-linear abelian category.

The functor $\omega_{\mathrm{dR}}\colon\Mot(B(\mathbb{F}))\rightarrow
\Mod_{B(\mathbb{F})}$ has a canonical factorization into
\[
\begin{CD}
\Mot(B(\mathbb{F}))@>\omega_{\mathrm{crys}}>>{\mathsf{R}_p}@>\text{\textrm{forget}}>>\Mod_{
B(\mathbb{F})};
\end{CD}
\]
and it even factors through $\MF_{B(\mathbb{F})}$ on a large subcategory of
$\Mot(B(\mathbb{F}))$.

\begin{definition}
\label{a28}A \emph{crystalline} $p$\emph{-integral structure\/} on a
$G$-motive $M$ over $B(\mathbb{F})$ is an exact tensor functor $\underline
{\Lambda}_{\mathrm{crys}}\colon\Rep_{\mathbb{Z}_{p}}(G)\rightarrow
\MF_{W(\mathbb{F})}$ such that $\underline{\Lambda}_{\mathrm{crys}}\left(
\xi\right)  $ is a $W(\mathbb{F})$-lattice in $(\omega_{\mathrm{dR}}\circ
M(p))(\xi_{\mathbb{Q}_{p}})$ for all $\xi$ in $\Rep_{\mathbb{Z}_{p}}(G_{p})$.
\end{definition}

The condition means that $\underline{\Lambda}_{\mathrm{crys}}\left(
\xi\right)  \subset(\omega_{\mathrm{dR}}\circ M(p))(\xi_{\mathbb{Q}_{p}})$ for
all $\xi$, and there is a commutative diagram%
\[
\xymatrixcolsep{4pc}\xymatrix{
\Rep_{B(\mathbb{F})}(G)\ar[r]^-{B(\mathbb{F})\otimes M}
&B(\mathbb{F})\otimes\Mot(B(\mathbb{F}))
\ar[r]^-{\omega_{\mathrm{crys}}} &\MF_{B(\mathbb{F})} \\
\Rep_{\mathbb{Z}_{p}}(G)\ar[u]\ar[rr]^{\underline{\Lambda}_{\mathrm{crys}}}
&  & \MF_{W}\ar[u].
}
\]

Recall that for any filtered module $N$, there is a canonical splitting
$\mu_{W}$ of the filtration on $N$, and that $\mu_{W}$ splits the filtration
on any strongly divisible submodule of $N$ (\cite{wintenberger1984}).

\begin{lemma}
\label{a29} Let $\underline{\Lambda}_{\mathrm{crys}}$ be a $p$-integral
crystalline structure on an admissible $G$-motive $M$ over $B(\mathbb{F})$.
Then there exists an isomorphism of tensor functors $W\otimes_{\mathbb{Z}_{p}%
}\Lambda\rightarrow\underline{\Lambda}_{\mathrm{crys}}$ carrying $\mu_{0}%
^{-1}$ into $\mu_{W}$.
\end{lemma}

\begin{proof}
It follows from (\ref{a1}) that there exists an isomorphism $\alpha\colon
W\otimes_{\mathbb{Z}_{p}}\Lambda\rightarrow\underline{\Lambda}_{\mathrm{crys}%
}$, uniquely determined up to composition with an element of $G(W)$. Let
$\mu^{\prime}$ be the cocharacter of $G_{W}$ mapped by $\alpha$ to $\mu_{W}$.
We have to show that $\mu^{\prime}$ is $G(W)$-conjugate to $\mu_{0}^{-1}$.

According to Lemma \ref{a27}, there exists an isomorphism $\beta\colon
B(\mathbb{F})\otimes_{\mathbb{Z}_{p}}\Lambda\rightarrow B(\mathbb{F}%
)\otimes_{W}\underline{\Lambda}_{\mathrm{crys}}$ carrying $\Filt(\mu_{0}%
^{-1})$ into $\Filt(\mu_{W})$. After possibly replacing $\beta$ with its
composite with an element of (a unipotent subgroup) of $G(B(\mathbb{F}))$ ---
see, for example, \cite{milne1990}, 1.7, --- we may assume that $\beta$ maps
$\mu_{0}^{-1}$ to $\mu_{W}$. Since $\beta$ can differ from $B(\mathbb{F}%
)\otimes_{W}\alpha$ only by an element of $G(B(\mathbb{F}))$, this shows that
$\mu^{\prime}$ is $G(B(\mathbb{F}))$-conjugate to $\mu_{0}^{-1}$.

Let $T$ be a maximal (split) torus of $G_{W}$ containing the image of
$\mu^{\prime}$. From its definition (see \ref{a11}), we know $\mu_{0}$ factors
through a specific torus $S\subset G_{W}$. According to \cite{demazureG1964},
XII 7.1, $T$ and $S$ will be conjugate locally for the \'{e}tale topology on
$\Spec W$, which in our case means that they are conjugate by an element of
$G(W)$. We may therefore suppose that $\mu^{\prime}$ and $\mu_{0}$ both factor
through $S$. But two characters of $S$ are $G(B(\mathbb{F}))$-conjugate if and
only if they are conjugate by an element of the Weyl group --- see for example
\cite{milne1992}, 1.7, --- and, because the hyperspecial point fixed by $G(W)$
lies in the apartment corresponding $S$, $G(W)$ contains a set of
representatives for the Weyl group (see \ref{a11}). Thus $\mu^{\prime}$ is
conjugate to $\mu_{0}^{-1}$ by an element of $G(W)$.
\end{proof}

\subsection{The integral comparison conjecture.}

Recall that, for an abelian variety over a field $k$ of characteristic zero,
$\mathcal{B}(A)$ denotes the $\mathbb{Q}$-algebra of Hodge classes on $A$
(absolute Hodge classes in the sense of \cite{deligne1982}). A Hodge tensor on
$A$ is an element of $\bigoplus\nolimits_{n\geq0}\mathcal{B}{}(A^{n})$. A
Hodge class on $A$ is a family $\gamma=(\gamma_{l})_{l}$ with $\gamma_{l}\in
H^{2\ast}(A_{k^{\mathrm{al}}},\mathbb{Q}_{l}(\ast))$ for $l\neq\infty$ and
$\gamma_{\infty}\in H_{\mathrm{dR}}^{2\ast}(A)(\ast)$.

Let $A$ be an abelian variety over a finite extension $K$ of $B(\mathbb{F})$
contained in $(\mathbb{Q}^{\mathrm{al}})_{v}$, and suppose that $K$ is
sufficiently large that $\mathcal{B}(A)\simeq\mathcal{B}(A_{K^{\mathrm{al}}}%
)$. Then%
\[
\mathcal{B}_{(p)}(A)\overset{\textup{{\tiny def}}}{=}\mathcal{B}(A)\cap
H^{2\ast}(A_{K^{\mathrm{al}}},\mathbb{Z}_{p})(\ast)
\]
is a $\mathbb{Z}_{(p)}$-lattice in $\mathcal{B}(A)$.

\begin{icconjecture}
\label{a30}If $A$ has good reduction, and so extends to an abelian scheme
$\mathcal{A}$ over $\mathcal{O}_{K}$, then, for $\gamma\in\mathcal{B}{}(A)$,%
\begin{equation}
\gamma_{p}\in H^{2\ast}(A_{k^{\mathrm{al}}},\mathbb{\mathbb{Z}{}}_{p}%
(\ast))\implies\gamma_{\mathrm{\infty}}\in H_{\mathrm{dR}}^{2\ast}%
(\mathcal{A})(\ast). \label{e5}%
\end{equation}

\end{icconjecture}

The following statement was conjectured in \cite{milne1995}.\footnote{More
precisely, \cite{milne1995}, Conjecture 0.1, p1, reads: \bquote Let $A$ be an
abelian scheme over $W(\mathbb{F})$. Let $\mathfrak{s}=(s_{i})_{i\in I}$ be a
family of Hodge tensors on $A$ including a polarization, and, for some fixed
inclusion $\tau\colon W(\mathbb{F})\hookrightarrow\mathbb{C}$, let $G$ be the
subgroup of $\GL(H^{1}((\tau A)(\mathbb{C}),\mathbb{Q}))$ fixing the $s_{i}$.
Assume that $G$ is reductive, and that the Zariski closure of $G$ in
$\GL(H^{1}(A_{B(\mathbb{F}{})},\mathbb{Z}_{p}))$ is hyperspecial. Then, for
some faithfully flat $\mathbb{Z}_{p}$-algebra $R$, there exists an isomorphism
of $W$-modules
\[
R\otimes_{\mathbb{Z}_{p}}H^{1}(A_{B(\mathbb{F})},\mathbb{Z}_{p})\rightarrow
R\otimes_{W}H_{\mathrm{dR}}^{1}(A)
\]
mapping the \'{e}tale component of each $s_{i}$ to the de Rham component. (In
fact, if there exists such an isomorphism for some faithfully flat
$\mathbb{Z}{}_{p}$-algebra, then there exists an isomorphism with
$R=W(\mathbb{F}{})$.)\equote}

\begin{theorem}
\label{a30a}Let $K$ be a finite extension of $B(\mathbb{F}{})$, and let $A$ be
an abelian variety over $B(\mathbb{F}{})$ with good reduction (so $A$ extends
to an abelian scheme over $\mathcal{O}{}_{K}$). Let $(s_{i})_{i\in I}$ be a
family of Hodge tensors for $A$ that are tensors for $H^{1}(A,\mathbb{Z}{}%
_{p})$ and that define a reductive subgroup of $\GL(H^{1}(A,\mathbb{Z}{}%
_{p}))$. Then there exists an isomorphism of $W(\mathbb{F}{})$-modules%
\[
W(\mathbb{F}{})\otimes_{\mathbb{Z}{}_{p}}H^{1}(A,\mathbb{Z}{}_{p})\rightarrow
H_{\mathrm{dR}}^{1}(\mathcal{A}{})
\]
mapping the \'{e}tale component of each $s_{i}$ to the de Rham component.
\end{theorem}

\begin{proof}
Apply \cite{kisin2009}, 1.3.6, or \cite{vasiu2003}.
\end{proof}

Note that Theorem \ref{a30a} implies that the Hodge tensors $s_{i}$ in the
statement satisfy (\ref{e5}).

\subsection{Construction of the map $\mathcal{M}(W(\mathbb{F}))\rightarrow
\mathcal{L}$}

Define $\Mot(W(\mathbb{F}))$ (respectively\ $\Mot_{(p)}(W(\mathbb{F}))$) to be
the tannakian subcategory of $\Mot(B(\mathbb{F}))$ (respectively\ $\Mot_{(p)}%
(B(\mathbb{F}))$) generated by abelian varieties over $B(\mathbb{F})$ with
potential good reduction. Note that Conjecture \ref{a30} implies that there is
an exact tensor functor $\omega_{\mathrm{dR}}\colon\Mot_{(p)}(W(\mathbb{F}%
))\rightarrow\Mod_{W(\mathbb{F})}$.

\begin{definition}
\label{a31}An admissible $G$-motive $M$ over $B(\mathbb{F})$ is \emph{special}
if it takes values in $\Mot(W(\mathbb{F}))$ and there exists an exact tensor
functor $M^{\prime}\colon\Rep(G)\rightarrow\CM(\mathbb{Q}^{\text{\textrm{al}}%
})$ making the following diagram commute:%
\[
\xymatrix{
& \CM(\mathbb{Q}^{\text{al}}) \ar[r] & \Mot(B(\mathbb{F})^{\mathrm{al}})\\
\Rep(G) \ar[r]^{M}\ar[ru]^{M^{\prime}} & \Mot(W(\mathbb{F})) \ar[r]^{\subset} & \Mot(B(\mathbb{F}))\ar[u]\\
}
\]

\end{definition}

We say that an equivalence class $m$ of admissible $G$-motives over
$B(\mathbb{F})$ is \emph{special} if it contains a special $G$-motive $M$, in
which case we let $\mathcal{M}(m)=\mathcal{M}(M^{\prime})$. We define
$\mathcal{M}(G,X)(W(\mathbb{F}))$ to be the subset $\bigsqcup_{m\text{
special}}\mathcal{M}(m)$ of $\mathcal{M}(G,X)(B(\mathbb{F}))$.

\begin{proposition}
\label{a32}Assume the integral comparison conjecture. Then there is a
canonical equivariant map%
\[
\mathcal{M}(G,X)(W(\mathbb{F}))\rightarrow\mathcal{L}(G,X).
\]
The map becomes surjective when we omit the $p$-component; that is, when we
replace $\mathcal{L}{}(G,X)$ with its quotient%
\[
\mathcal{\mathcal{L}{}}^{p}(G,X)\overset{\textup{{\tiny def}}}{=}%
\bigsqcup\nolimits_{m}\mathcal{L}{}^{p}(m),\quad\mathcal{\mathcal{L}{}}%
^{p}(m)\overset{\textup{{\tiny def}}}{=}\varprojlim_{M\in m}{}%
\mathcal{\mathcal{L}{}}^{p}(M),\quad\mathcal{\mathcal{L}{}}^{p}%
(M)=I(M)\backslash(X^{p}(M)/Z^{p}).
\]

\end{proposition}

\begin{proof}
Let $M$ be special. The composite%
\[
\Rep(G)\overset{M^{\prime}}{\longrightarrow}\CM(\mathbb{Q}^{\text{al}%
})\overset{R}{\longrightarrow}\Mot(\mathbb{F})
\]
is a special $G$-motive $\bar{M}$ over $\mathbb{F}$. Let $\eta\in X^{p}(M)$,
and let $\underline{\Lambda}_{p}\in X_{p}(M)$. Then $(M,\underline{\Lambda
}_{p})$ defines an exact tensor functor $\Rep(G)\rightarrow\Mot_{(p)}%
(W(\mathbb{F}{}))$, whose composite with $\omega_{\mathrm{dR}}$ is an element
of $X_{p}(\bar{M})$ (because of \ref{a27} and \ref{a29}). The specialization
map in \'{e}tale cohomology defines an isomorphism $\omega_{f}^{p}%
(M)\rightarrow\omega_{f}^{p}(\bar{M})$, and so $\eta$ defines an element of
$X^{p}(\bar{M})$. Therefore, we have a map $\mathcal{M}(M)\rightarrow
\mathcal{M}(\bar{M})$. On combining these maps for different $M$, we obtain a
map $\mathcal{M}(W(\mathbb{F}))\rightarrow\mathcal{L}$. The second statement
is obvious from the definitions.
\end{proof}

\begin{question}
\label{a33}Is there an elementary proof that every admissible $G$-motive over
$B(\mathbb{F})$ that takes values in $\Mot(W(\mathbb{F}))$ is special? Perhaps
this can be deduced from an analogue of Proposition \ref{a19}. A positive
answer would give an elementary proof of the special-points conjecture (see
the proof of Theorem \ref{a46}).
\end{question}

\begin{remark}
\label{a34}When $(G,X)$ is of Hodge type, then Proposition \ref{a32} follows
from Theorem \ref{a30a}, i.e., it is not necessary to assume the integral
comparison conjecture. To see this, choose a defining representation and
tensors for $G$.
\end{remark}

\section{Proof of Conjecture LR+ for certain Shimura varieties}

Let $(G,X)$ be a Shimura $p$-datum satisfying (SV1-4,6), and let $\Sh_{p}$ be
a canonical integral model. Consider the diagram\begin{equation}\begin{gathered}
\bfig
\node a(0,500)[\mathcal{M}(W(\mathbb{F}))]
\node b(300,500)[\subset]
\node c(600,500)[\mathcal{M}(B(\mathbb{F}))]
\node d(1300,500)[\Sh_{p}(B(\mathbb{F}))]
\node e(1600,500)[\simeq]
\node f(1900,500)[\Sh_{p}(W(\mathbb{F}))]
\node g(0,0)[\mathcal{L}]
\node h(1900,0)[\Sh_{p}(\mathbb{F})]
\arrow/{}/[a`b;{}]
\arrow/{}/[b`c;{}]
\arrow[c`d;\simeq]
\arrow/{}/[d`e;{}]
\arrow/{}/[e`f;{}]
\arrow[a`g;{}]
\arrow/{-->}/[g`h;\textrm{?}]
\arrow/{->>}/[f`h;{}]
\efig
\end{gathered}\label{e3}\end{equation}

\noindent in which the isomorphism on the top row is the map $\alpha
(B(\mathbb{F}{}))$ of Theorem \ref{a25}. We say that Conjecture LR holds for
$(G,X)$ if there exists an equivariant isomorphism $\mathcal{L}\rightarrow
\Sh_{p}(\mathbb{F})$ making the diagram commute.

\subsection{\noindent Siegel modular varieties}

Let $(G(\psi),X(\psi))$ be the Shimura datum attached to a symplectic space
$(V,\psi)$ over $\mathbb{Q}$. Thus $G(\psi)=\GSp(\psi)$ and $X(\psi)$ consists
of the Hodge structures $h$ on $V_{\mathbb{R}}$ for which $(x,y)\mapsto
\psi(x,h(i)y)$ is definite (either positive or negative). For any
$\mathbb{Z}_{(p)}$-lattice $\Lambda$ in $V$ such that $\psi$ restricts to a
perfect $\mathbb{Z}_{(p)}$-valued pairing on $\Lambda$, the subgroup $G$ of
$G(\psi)$ stabilizing $\Lambda$ is a reductive group over $\mathbb{Z}{}_{(p)}$
with generic fibre $G(\psi)$. Thus $(G,X(\psi))$ is a Shimura $p$-datum. Any
Shimura $p$-datum arising in this way will be called a \emph{Siegel }%
$p$\emph{-datum}$.$ In this subsection, we prove Conjecture LR for Siegel $p$-data.

\begin{lemma}
\label{a37}Let $(A,\lambda)$ be a polarized motive over $\mathbb{F}$. For any
$\mathbb{Q}$-algebra $R$, let $H(R)$ be the group of automorphisms of $A$ as
an object of $\Mot(\mathbb{F})_{(R)}$ fixing $\lambda$. Then $H$ is an
algebraic group over $\mathbb{Q}$ satisfying the Hasse principle for $H^{1}$.
\end{lemma}

\begin{proof}
We may assume that $A$ is isotypic. Let $L=\End(A)$, let $E$ be the centre of
$L$, and let $^{\dagger}$ denote the involution of $L$ defined by $\lambda$.
For any $\mathbb{Q}$-algebra $R,$%
\[
H(R)=\{a\in L\otimes_{\mathbb{Q}{}}R\mid aa^{\dagger}=1\}.
\]
There are two cases to consider: (a) $E=\mathbb{Q}$; (b) $E$ is a CM-field
(cf.\ \cite{milne1994}, 2.16). Choose a fibre functor $\bar{\omega}$ over
$\mathbb{Q}^{\mathrm{al}}$, and let $V=\bar{\omega}(A)$. In case (a),
$\End(A)_{\mathbb{Q}^{\mathrm{al}}}=\End(V)$ and $H_{\mathbb{Q}^{\mathrm{al}}%
}$ is the symplectic group attached to the alternating form on $V$ defined by
$\lambda$. Hence $H$ is a simply connected semisimple group, and so it
satisfies the Hasse principle for $H^{1}$. In case (b), $V$ decomposes into a
direct sum of spaces $W\oplus W^{\vee}$. Correspondingly, $H_{\mathbb{Q}%
^{\mathrm{al}}}$ decomposes as a product of groups $G_{W}$, where
$G_{W}=\End(W)\oplus\End(W^{\vee})$. Moreover, $(a,b)^{\dagger}%
=(b^{\mathrm{tr}},a)$, and so $G_{W}\simeq\GL_{W}$. From this, one deduces
that $H=\Res_{F/\mathbb{Q}}H^{\prime}$ where $H^{\prime}$ is an algebraic
group over the largest totally real subfield $F$ of $E$; the reduce norm
defines a homomorphism $H^{\prime}\rightarrow\mathbb{G}_{m}$ and the kernel is
a form of $\SL_{n}$ for some $n$. From the diagram%
\[
\minCDarrowwidth25pt\begin{CD}
@.F^{\times} @>>> H^{1}(F,H^{\prime\mathrm{der}}) @>>> H^{1}(F,H^{\prime})
@>>> 0\\
@.@V{\text{\phantom{g}dense}}V{\text{image}}V@VV{\simeq}V@VVV\\
\prod\limits_{v|\infty}H^{\prime}(F_{v}) @>{\text{open}}>{\text{image}}>
\prod\limits_{v|\infty}F_{v}^{\times} @>>> \prod\limits_{v|\infty}H^{1}%
(F_{v},H^{\prime\mathrm{der}}) @>>> \prod\limits_{v}H^{1}(F_{v},H^{\prime
})
\end{CD}
\]
we see that $H^{\prime}$ satisfies the Hasse principle over $F$, and so $H$
satisfies the Hasse principle over $\mathbb{Q}$.
\end{proof}

\begin{lemma}
\label{a38}Let $(A,\lambda)$ and $(A^{\prime},\lambda^{\prime})$ be polarized
CM abelian varieties over $\mathbb{Q}^{\mathrm{al}}$. If there exists an
isogeny $A_{0}^{\prime}\rightarrow A_{0}$ sending $\lambda_{0}^{\prime}$ to
$\lambda_{0}$, then there exists an isomorphism $\hbar^{1}(A)\rightarrow
\hbar^{1}(A^{\prime})$ sending $\hbar(\lambda)$ to $\hbar(\lambda^{\prime})$.
\end{lemma}

\begin{proof}
For a $\mathbb{Q}$-algebra $R$, let $\mathcal{F}(R)$ be the set of
isomorphisms from $\hbar^{1}A$ to $\hbar^{1}A^{\prime}$ in the category
$\Mot(\mathbb{F})_{(R)}$ sending $\hbar\lambda$ to $\hbar\lambda^{\prime}$. We
have to show that $\mathcal{F}(\mathbb{Q})$ is nonempty. Clearly $\mathcal{F}$
is a torsor for the algebraic group $H$ in Lemma \ref{a37}, and so it suffices
to show that $\mathcal{F}(\mathbb{Q}_{l})$ is nonempty for all $l$.

Let $l$ be a prime $\neq p,\infty$, and let $\mathsf{R}_{l}^{\prime}$ be the
category of finite-dimensional $\mathbb{Q}_{l}$-vector spaces equipped with a
germ of a Frobenius map$.$ Then $\zeta_{l}$ defines a fully faithful functor
$\Mot(\mathbb{F})_{(\mathbb{Q}_{l})}\rightarrow\mathsf{R}_{l}^{\prime}$
(cf.\ \cite{milne1994}, 3.7). Moreover, $\zeta_{l}(\hbar^{1}A)=\xi_{l}%
(h^{1}A)$ (by definition), and $\xi_{l}(h^{1}A)=H_{\mathrm{et}}^{1}%
(A,\mathbb{Q}_{l})\simeq H_{\mathrm{et}}^{1}(A_{0},\mathbb{Q}_{l})$ (as
objects of $\mathsf{R}_{l}^{\prime}$). An isogeny $A_{0}^{\prime}\rightarrow
A_{0}$ sending $\lambda_{0}^{\prime}$ to $\lambda_{0}^{\prime}$ defines an
isomorphism $H_{\mathrm{et}}^{1}(A_{0},\mathbb{Q}_{l})\rightarrow
H_{\mathrm{et}}^{1}(A_{0}^{\prime},\mathbb{Q}_{l})$ sending $H_{\mathrm{et}%
}^{1}(\lambda_{0})$ to $H_{\mathrm{et}}^{1}(\lambda_{0}^{\prime})$, and hence
an element of $\mathcal{F}(\mathbb{Q}_{l})$.

For $p$, we define $\mathsf{R}_{p}^{\prime}$ to the category of finite
dimensional vector spaces $V$ over $\mathbb{Q}_{p}^{\mathrm{un}}$ equipped
with a $\sigma$-linear isomorphism $V\rightarrow V$. Then $\zeta_{l}$ defines
a fully faithful functor $\Mot(\mathbb{F})_{(\mathbb{Q}_{p})}\rightarrow
\mathsf{R}_{p}^{\prime}$. The same argument as in the last paragraph shows
that $\mathcal{F}(\mathbb{Q}_{p})$ is nonempty.

For $\infty$, we need to use the categories of Lefschetz motives
$\LCM(\mathbb{Q}^{\text{al}})$ and $\LMot(\mathbb{F})$ (see \cite{milne1999lm}%
). Each of these categories has a canonical polarization for which the
geometric polarizations are positive, and the quotient map $\LCM(\mathbb{Q}%
^{\text{al}})\rightarrow\LMot(\mathbb{F})$ preserves the polarizations (see
\cite{milne2002}, 3.7). These polarizations define (isomorphism classes of)
functors from $\LCM(\mathbb{Q}^{\text{al}})$ and $\LMot(\mathbb{F})$ to
$\mathsf{R}_{\infty}$ (see \cite{deligneM1982}, 5.20). We can choose the
functor on $\LCM(\mathbb{Q}^{\text{al}})$ to be the composite $\LCM(\mathbb{Q}%
^{\text{al}})\rightarrow\CM(\mathbb{Q}^{\mathrm{al}})\rightarrow
\mathsf{R}_{\infty}$; then we can choose the functor on $\LMot(\mathbb{F})$ to
be compatible, via the quotient map, with that on $\LCM(\mathbb{Q}^{\text{al}%
})$. The functor $\zeta_{\infty}$ defines a fully faithful functor from
$\Mot(\mathbb{F})_{(\mathbb{R})}$ to the category of objects of $\mathsf{R}%
_{\infty}$ equipped with a germ of a Frobenius operator (cf.\ \cite{milne1994}%
, 3.6). Now, the same argument as before shows that $\mathcal{F}(\mathbb{R})$
is nonempty.
\end{proof}

\begin{theorem}
\label{a39}Let $(G,X)$ be a Siegel $p$-datum. The map $\mathcal{M}%
(W(\mathbb{F}))\rightarrow\Sh_{p}(\mathbb{F})$ of (\ref{e3}) induces an
isomorphism $\mathcal{L}\rightarrow$ $\Sh_{p}(\mathbb{F})$.
\end{theorem}

\begin{proof}
The group $G$ has a natural Tate triple structure $(w,t)$, and we prefer to
work with $(G,w,t)$-objects (see p\pageref{a8}). Now fix a polarized CM
abelian variety $(A,\lambda)$ over $\mathbb{Q}^{\mathrm{al}}$, and let $M$ be
the corresponding $(G,w,t)$-motive over $B(\mathbb{F})$ (see \ref{a8}). It
follows from Lemma \ref{a38} and deformation theory (starting from the
Serre-Tate theorem; cf \cite{norman1981}) that the fibres of the maps
$\mathcal{M}(M)\rightarrow\mathcal{L}$ and $\mathcal{M}(M)\rightarrow
\Sh_{p}(\mathbb{F})$ are equal. Therefore, we get a commutative diagram%
\[
\begin{CD}
\mathcal{M}(M) @>{\textrm{bijection}}>>{}^{\prime}\Sh_{p}(B(\mathbb{F}))\\
@VV{\mathrm{onto}}V@VV{\mathrm{onto}}V\\
{}^{\prime}\mathcal{L}@>{\mathrm{bijection}}>>{}^{\prime}\Sh_{p}(\mathbb{F}{})%
\end{CD}
\]
where a prime denotes the image of $\mathcal{M}{}(M)$. Using the second
statement of Proposition \ref{a32}, one sees that the lower arrow extends
canonically to an injection $\mathcal{L}{}\rightarrow\Sh_{p}(\mathbb{F}{})$,
which is a bijection by Zink's theorem (\cite{zink1983}, 2.7).
\end{proof}

\subsection{Shimura subvarieties}

Recall that a map $(G,X)\rightarrow(G^{\prime},X^{\prime})$ of Shimura data
defines a morphism $\Sh(G,X)\rightarrow\Sh(G^{\prime},X^{\prime})$ of Shimura
varieties over $\mathbb{C}$. Deligne (1971, 1.15) shows that the morphism is a
closed immersion if $G\rightarrow G^{\prime}$ is injective.

His argument proves a similar statement for Shimura $p$-data.\footnote{This is
written out, for example, in \cite{giguere1998}, 2.1.9.1 and in
\cite{kisin2009}, 2.1.2.} In particular, if $(G,X)\rightarrow(G^{\prime
},X^{\prime})$ is a map of Shimura $p$-data such that $G\rightarrow G^{\prime
}$ is injective (as a map of group schemes over $\mathbb{Z}{}_{(p)}$), then
$\Sh_{p}(G,X)(B(\mathbb{F}{}))\rightarrow$ $\Sh_{p}(G^{\prime},X^{\prime
})(B(\mathbb{F}{}))$ is injective. Similarly, $\mathcal{L}{}(G,X)\rightarrow
\mathcal{L}{}(G^{\prime},X^{\prime})$ is injective. Write $?^{\prime}$ for
$?(G^{\prime},X^{\prime})$. For suitable integral models of the Shimura
varieties, there will be a homomorphism of diagrams%
\[
\begin{CD}
\mathcal{M}(W(\mathbb{F}))@>{\subset}>>\Sh_{p}(W(\mathbb{F}))\\
@VVV@VVV\\
\mathcal{L}@.\Sh_{p}(\mathbb{F})
\end{CD}
\quad\longrightarrow\quad\begin{CD}
\mathcal{M}^{\prime}(W(\mathbb{F}))@>{\subset}>>\Sh_{p}^{\prime}(W(\mathbb{F}))\\
@VVV@VVV\\
\mathcal{L}^{\prime}@.\Sh_{p}^{\prime}(\mathbb{F}).
\end{CD}
\]

\begin{proposition}
\label{a40}Assume that the map $\Sh_{p}(\mathbb{F}{})\rightarrow
\Sh_{p}^{\prime}(\mathbb{F}{})$ is injective. If the map $\mathcal{M}%
{}^{\prime}(W(\mathbb{F}{}))\rightarrow\Sh_{p}^{\prime}(\mathbb{F}{})$ factors
through $^{\prime}\mathcal{L}{}^{\prime}$ and defines a bijection $^{\prime
}\mathcal{L}{}^{\prime}\rightarrow{}^{\prime}\Sh_{p}^{\prime}(\mathbb{F}{})$,
then the same statement is true for $\mathcal{\mathcal{M}{}}(W(\mathbb{F}%
{}))\rightarrow\Sh_{p}(\mathbb{F}{})$.
\end{proposition}

\begin{proof}
This is obvious from the above statements.
\end{proof}

\subsection{Shimura $p$-data of Hodge type}

We say that a Shimura $p$-datum $(G,X)$ is of \emph{Hodge type} if there
exists a Siegel $p$-datum $(G^{\prime},X^{\prime})$ and a map
$(G,X)\rightarrow(G^{\prime},X^{\prime})$ with $G\rightarrow G^{\prime}$
injective. We let $\Sh_{p}$ and $\Sh_{p}^{\prime}$ denote the canonical
integral models.

\begin{theorem}
\label{a44}Let $(G,X)$ be a Shimura $p$-datum of Hodge type, and assume that
the map $\Sh_{p}(\mathbb{F}{})\rightarrow\Sh_{p}^{\prime}(\mathbb{F}{})$ is
injective for some Siegel embedding. Then the map $\mathcal{\mathcal{M}%
}(W(\mathbb{F}))\rightarrow\Sh_{p}(\mathbb{F})$ of (\ref{e3}), p\pageref{e3},
defines an injection $\mathcal{L}\rightarrow$ $\Sh_{p}(\mathbb{F})$, which is
surjective if the special-points conjecture holds for $(G,X)$.
\end{theorem}

\begin{proof}
Theorem \ref{a39} and Proposition \ref{a40} show that the map
$\mathcal{\mathcal{M}}(W(\mathbb{F}))\rightarrow\Sh_{p}(\mathbb{F})$ defines a
bijective map $^{\prime}\mathcal{L}{}\rightarrow{} ^{\prime}\Sh_{p}%
(\mathbb{F}{})$, which can be extended to an injection map $\mathcal{L}%
{}\rightarrow\Sh_{p}(\mathbb{F}{})$ using the second part of Proposition
\ref{a32}. If the special-points conjecture is true for $(G,X)$, this map is surjective.
\end{proof}

\begin{plain}
We deduce:

\begin{enumerate}
\item Conjecture LR+ is true for all Shimura varieties of PEL-type (by Zink's
theorem; in this case, it is known that the canonical integral model has the
property that $\Sh_{p}(\mathbb{F}{})\rightarrow\Sh_{p}^{\prime}(\mathbb{F}{})$
is injective).

\item Conjecture LRo is true for a Shimura variety of $p$-Hodge type if (and
only if) the special-points conjecture holds for the Shimura variety
(Conjecture LRo allows us to take the integral model of $\Sh_{p}$ to be the
closure of $\Sh_{p}$ in the integral model of $\Sh_{p}^{\prime}$).

\item Conjecture LR+ is true for a Shimura variety of $p$-Hodge type if the
special-points conjecture holds for it and canonical integral model has the
property that $\Sh_{p}(\mathbb{F}{})\rightarrow\Sh_{p}^{\prime}(\mathbb{F}{})$
is injective (the canonical integral model is not always known to have this
last property, because present constructions of it require a normalization;
according to Vasiu, this should not be necessary).\footnote{Note that, as the
canonical integral model is currently constructed, the map $\Sh_{p}%
(W(\mathbb{F}{}))\rightarrow\Sh_{p}^{\prime}(W(\mathbb{F}{}))$ is injective
(because it is a submap of the injective map $\Sh_{p}(B(\mathbb{F}%
{}))\rightarrow\Sh_{p}^{\prime}(B(\mathbb{F}{}))$).}
\end{enumerate}
\end{plain}

\subsection{Complements}

\subsubsection{Restatement in terms of groupoids}

Choose a $\mathbb{Q}^{\mathrm{al}}$-valued fibre functor $\bar{\omega}$ on
$\Mot(\mathbb{F})$ and isomorphisms
\begin{equation}
\mathbb{C}\otimes_{\mathbb{Q}^{\mathrm{al}}}\bar{\omega}\rightarrow
\omega_{\infty}^{\mathbb{F}},\quad\mathbb{Q}_{\ell}^{\mathrm{al}}%
\otimes_{\mathbb{Q}^{\mathrm{al}}}\bar{\omega}\rightarrow\mathbb{Q}_{\ell
}^{\mathrm{al}}\otimes_{\mathbb{Q}_{\ell}}\omega_{\ell}^{\mathbb{F}},\quad
B(\mathbb{F}{})^{\mathrm{al}}\otimes_{\mathbb{Q}^{\mathrm{al}}}\bar{\omega
}\rightarrow B(\mathbb{F}{})^{\mathrm{al}}\otimes_{B(\mathbb{F})}\omega
_{p}^{\mathbb{F}}. \label{eq3}%
\end{equation}
Then $\mathfrak{P}\overset{\textup{{\tiny def}}}{=}\underline{\Aut}%
_{\mathbb{Q}}^{\otimes}(\bar{\omega})$ is a transitive affine $\mathbb{Q}%
^{\mathrm{al}}/\mathbb{Q}$-groupoid (\cite{deligne1990}, 1.11, 1.12). Let
$\mathfrak{G}_{l}$ be the groupoid attached to the category $\mathsf{R}_{l}$
and its canonical (forgetful) fibre functor. Then%
\[
\underline{\Aut}_{\mathbb{Q}_{l}}^{\otimes}(\omega_{l}^{\mathbb{F}}%
)\simeq\mathfrak{G}_{l},\quad l=2,3,5,\ldots,\infty,
\]
and so the isomorphisms (\ref{eq3}) define homomorphisms $\check{\zeta}%
_{l}\colon\mathfrak{G}_{l}\rightarrow\mathfrak{P}(l)$ where $\mathfrak{P}(l)$
is the $\mathbb{Q}_{l}^{\mathrm{al}}/\mathbb{Q}_{l}$-groupoid obtained from
$\mathfrak{P}$ by extension of scalars. The kernel $\mathfrak{G}_{l}^{\Delta}$
of $\mathfrak{G}_{l}$ is $\mathbb{G}_{m}$ for $l=\infty$, $1$ for $l\neq
p,\infty$, and $\mathbb{G}$ for $l=p$.

\begin{proposition}
\label{a45}The system $(\mathfrak{P},(\check{\zeta}_{l})_{l=2,3,5,\ldots
,\infty})$ satisfies the following conditions

\begin{enumerate}
\item $(\mathfrak{P}^{\Delta},\check{\zeta}_{p}^{\Delta},\check{\zeta}%
_{\infty}^{\Delta})=(P,z_{p},z_{\infty});$

\item the morphisms $\check{\zeta}_{\ell}$ for $\ell\neq p,\infty$ are defined
by a section of $\mathfrak{P}$ over $\overline{\mathbb{A}_{f}^{p}}%
\otimes_{\mathbb{Q}}\overline{\mathbb{A}_{f}^{p}}$ where $\overline
{\mathbb{A}_{f}^{p}}$ is the image of $\mathbb{Q}^{\mathrm{al}}\otimes
\mathbb{A}_{f}^{p}$ in $\prod\nolimits_{\ell\neq p,\infty}\mathbb{Q}_{\ell
}^{\mathrm{al}}$.
\end{enumerate}
\end{proposition}

\begin{proof}
Straightforward (cf. \cite{milne2003}, \S 6).
\end{proof}

Thus $(\mathfrak{P},(\check{\zeta}_{l}))$ is a pseudomotivic groupoid in the
sense of \cite{milne1992}, 3.27 (corrected in \cite{reimann1997}, p120), which
is essentially the same as the object that \citet{langlandsR1987} call a
\textquotedblleft pseudomotivische Galoisgruppe\textquotedblright\ . On
restating Conjecture \ref{a15} in terms of $(\mathfrak{P},(\check{\zeta}%
_{l}))$, we recover Conjecture 5.e, p169, of \cite{langlandsR1987} in the good
reduction case; see also \cite{milne1992}, 4.4, and \cite{reimann1997},
Appendix B3.

Langlands and Rapoport also state their conjecture for Shimura varieties whose
weight is not rational (that is, which fail SV4) in terms of a
\textquotedblleft quasimotivische Galoisgruppe\textquotedblright, but, as Pfau
and Reimann have pointed out, their definition of this is incorrect. Following
Pfau, we define a \emph{quasimotivic groupoid }by\footnote{Those who don't
wish to assume that their Shimura varieties satisfy SV6, will need to replace
$(\mathbb{G}_{m})_{\mathbb{Q}^{\mathrm{cm}}/\mathbb{Q}}$ with $(\mathbb{G}%
_{m})_{\mathbb{Q}^{\mathrm{al}}/\mathbb{Q}}$.}%
\[
\mathfrak{Q}=\mathfrak{P}\times_{S}(\mathbb{G}_{m})_{\mathbb{Q}^{\mathrm{cm}%
}/\mathbb{Q}}.
\]
Then it is possible Conjecture LR+ for all Shimura varieties in terms of
$\mathfrak{Q}$.

\subsubsection{Non simply connected derived groups}

Langlands and Rapoport originally stated their conjecture only for pairs
$(G,X)$ with $G^{\mathrm{der}}$ simply connected (in fact, their statement
becomes false without this condition). In \cite{milne1992}, the conjecture is
restated so that it applies to all Shimura varieties, and the condition that
the bijection be equivariant for $Z(\mathbb{Q}_{p})$ is added. For the
improved conjecture, the following statement is proved (ibid. 4.19):

\begin{quote}
Let $(G,X)\rightarrow(G^{\prime},X^{\prime})$ be a morphism of Shimura
$p$-data with $G_{\mathbb{Q}}\rightarrow G_{\mathbb{Q}}^{\prime}$ an isogeny;
if the improved Langlands-Rapoport conjecture is true for $(G,X)$, then it is
true for $(G^{\prime},X^{\prime})$.
\end{quote}

\noindent Similar arguments prove this for the Conjecture LR+. Thus, once one
knows Conjecture LR+ to be true for some Shimura $p$-data, one obtains it for
many more.

\subsubsection{Changing the centre of $G$}

\citet{pfau1993, pfau1996a, pfau1996b} stated a \textquotedblleft
refined\textquotedblright\ Langlands-Rapoport conjecture, and he proved the
following statement:

\begin{quote}
Let $(G,X)$ and $(G^{\prime},X^{\prime})$ be Shimura $p$-data whose associated
\emph{connected} Shimura $p$-data are isomorphic; if the refined
Langlands-Rapoport conjecture is true for one of the Shimura varieties, then
it is true for both.
\end{quote}

\noindent Again, similar arguments prove this statement Conjecture LR+. In
fact, the arguments become somewhat simpler. Thus Theorem \ref{a44} implies
Conjecture LR+ for many Shimura varieties whose weight is not rational.

\subsubsection{Shimura varieties of abelian type}

The above statements almost suffice to prove Conjecture LR+ for all Shimura
varieties of abelian type (assuming the special-points conjecture). The main
obstacle is that, in Theorem \ref{a44} we required that the Shimura $p$-datum
be of Hodge type, whereas we need to know the conjecture for all $(G,X)$ such
that $(G_{\mathbb{Q}},X)$ is of Hodge type. In other words, in Theorem
\ref{a44}, we required that there exist an embedding $(G_{\mathbb{Q}%
},X)\hookrightarrow(G(\psi),X(\psi))$ such that $G(\mathbb{Z}_{p})$ maps into
a hyperspecial subgroup of $G(\psi)$; we need to prove the theorem without the
last condition (or prove that it always holds).

\subsubsection{General Shimura varieties}

It is natural to pose the following problem:

\begin{quote}
Let $(G,X)$ be a Shimura $p$-datum (whose weight is defined over $\mathbb{Q}$,
if you wish). Define (in a natural way) an $\mathbb{F}$-scheme $L$ with a
continuous action of $G(\mathbb{A}_{f}^{p})\times Z(\mathbb{Q}_{p})$ such that
$L(\mathbb{F})=\mathcal{L}(G,X)$.
\end{quote}

\noindent Once this problem has been solved, it becomes possible to state the
following conjecture:

\begin{quote}
Show that there exists an equivariant isomorphism of $\mathbb{F}$-schemes
$L\rightarrow\Sh_{p}$.
\end{quote}

\noindent Once the problem has been solved for all Shimura varieties and the
conjecture has been proved for (certain) Shimura varieties of type $A_{n}$, it
should be possible to deduce the conjecture for all Shimura varieties by the
methods of \cite{milne1983} and \cite{borovoi1984, borovoi1987}.

\subsection{The rationality conjecture}

We refer to \cite{milne2009}, 4.1, for the statement of the rationality conjecture.

If the rationality conjecture is true for all CM abelian varieties, then there
is a unique good theory of rational Tate classes $A\mapsto\mathcal{R}(A)$ on
abelian varieties over $\mathbb{F}$, and we can define $\Mot(\mathbb{F})$ to
be the category whose objects are triples $(A,e,m)$ with $A$ an abelian
variety \emph{over }$\mathbb{F}$ and $e$ an idempotent in $\mathcal{R}^{\dim
A}(A\times A)$. The reduction functor realizes $\Mot(\mathbb{F}{})$ as a
quotient of the tannakian category $\CM(\mathbb{Q}^{\text{al}})$, and so it
defines a fibre functor $\omega_{0}$ on $\CM(\mathbb{Q}^{\text{al}})^{P}$. The
quotient category defined by $\omega_{0}$ is, of course, just $\Mot(\mathbb{F}%
)$. The advantage now is that we know directly that $\Mot(\mathbb{F})$
contains the motives of abelian varieties over $\mathbb{F}{}$.

\begin{theorem}
\label{a46}If the special-points conjecture is true, then the rationality
conjecture for CM abelian varieties implies the rationality conjecture for all
abelian varieties whose Mumford-Tate group is unramified at $p$. Conversely,
if the rationality conjecture is true for all abelian varieties, then the
special-points conjecture is true for Shimura varieties of Hodge type with
simply connected derived group (and hence for all Shimura varieties of abelian
type A, B, or C).
\end{theorem}

\begin{proof}
Let $A$ be an abelian variety over $\mathbb{Q}^{\mathrm{al}}$ with good
reduction at $v$ to an abelian variety $A_{0}$ over $\mathbb{F}$, and let
$\gamma$ be a Hodge class on $A$ and $\delta$ a Lefschetz class on $A_{0}$. If
the special points conjecture is true, then there exists a CM abelian variety
$A^{\prime}$ over $\mathbb{Q}^{\mathrm{al}}$, a Hodge class $\gamma^{\prime}$
on $A^{\prime}$, and an isogeny $\alpha\colon A_{0}^{\prime}\rightarrow A_{0}$
sending $\gamma_{0}^{\prime}$ to $\gamma_{0}$. Then%
\[
\langle\gamma_{0}\cdot\delta\rangle\in\langle\gamma_{0}^{\prime}\cdot
\alpha^{\ast}\delta\rangle\mathbb{Q}{}.
\]
If the rationality conjecture is true for $A^{\prime}$, then $\langle
\gamma_{0}^{\prime}\cdot\alpha^{\ast}\delta\rangle\in\mathbb{Q}{}$, and so
$\langle\gamma_{0}\cdot\delta\rangle\in\mathbb{Q}{}$.

Conversely, if the rationality conjecture is true for all abelian varieties,
then the reduction functor is defined on the tannakian subcategory
$\Mot(W(\mathbb{F}{}))$ of $\Mot(B(\mathbb{F}{}))$ generated by abelian
varieties over $B(\mathbb{F}{})$ with good reduction. A point $P$ of
$\Sh_{p}(\mathbb{F}{})$ arises from an element $[M,\eta,\underline{\Lambda
}_{p}]$ of $\mathcal{M}{}(B(\mathbb{F}{}))$ (see (\ref{e3}), p\pageref{e3}).
If $(G_{\mathbb{Q}{}},X)$ is of Hodge type, then $M$ takes values in
$\Mot(W(\mathbb{F}{}))$, and the composite%
\[
\Rep_{\mathbb{Q}{}}(G)\rightarrow\Mot(W(\mathbb{F}{}))\rightarrow
\Mot(\mathbb{F}{})
\]
satisfies the conditions (a,b,c,d) of Proposition \ref{a19}. Therefore, if
$G^{\mathrm{der}}$ is simply connected, then $M$ is special, which implies
that $P$ lifts to a special point.
\end{proof}

If we knew the rationality conjecture was true, this would open up the
possibility of extending the motivic moduli description of the points on a
Shimura variety of abelian type and rational weight from characteristic zero
to characteristic $p$.

So long as the Tate conjecture remains inaccessible, the rationality
conjecture is the most important problem in the theory of abelian varieties
over finite fields.

\bibliographystyle{cbe}
\bibliography{D:/Current/LRxrefs}

\bigskip\noindent\begin{minipage}[t]{2in}
Mathematics Department,\\
University of Michigan,\\
Ann Arbor, MI\ 48109, USA.
\end{minipage}\hfill\begin{minipage}[t]{1.7in}
jmilne at umich dot edu\\
\url{www.jmilne.org/math/}
\end{minipage}

\end{document}